\documentclass[reqno,12pt,letterpaper]{amsart}

\usepackage{amsmath,amssymb,amsthm,graphicx,mathrsfs,url,bm,subfigure,accents,paralist,xargs,slashed,enumitem,mathtools,array}
\usepackage[usenames,dvipsnames]{xcolor}
\usepackage[colorlinks=true,linkcolor=Red,citecolor=Green]{hyperref}
\usepackage[colorinlistoftodos,prependcaption,textsize=tiny]{todonotes}
\presetkeys{todonotes}{inline}{}

\numberwithin{equation}{section}

\newtheorem{theo}{Theorem}
\newtheorem{prop}{Proposition}[section]
\newtheorem{lemm}[prop]{Lemma}

\theoremstyle{definition}
\newtheorem{defi}[prop]{Definition}

\DeclareMathOperator{\supp}{supp}

\newcommand{\RR}{\mathbb{R}}
\newcommand{\CC}{\mathbb{C}}

\newcommand{\KK}{\mathbb{K}}

\newcommand{\loc}{\mathrm{loc}}

\newcommand{\avail}{V_{\mathrm{a}}}
\newcommand{\req}{V_{\mathrm{r}}}

\DeclareMathOperator{\rank}{rank}
\newcommand{\nrank}{\operatorname{rank}_{\KK(s)}}
\DeclareMathOperator{\range}{range}
\DeclareMathOperator{\disc}{disc}
\DeclareMathOperator{\diag}{diag}

\renewcommand{\Re}{\operatorname{Re}}

\title[Frequency criteria for exponential stability]
{Frequency criteria for exponential stability}

\author{Oran Gannot}
\email{ogannot@berkeley.edu}

\begin{document} 
	
	\begin{abstract} 
We discuss some frequency-domain criteria for the exponential stability of nonlinear feedback systems based on dissipativity theory. Applications are given to convergence rates for certain perturbations of the damped harmonic oscillator.
	\end{abstract}
	
	\maketitle
	
	\thispagestyle{empty}

	\section{Introduction}
\subsection{Overview} \label{subsect:overview}

This paper concerns exponential stability for nonlinear feedback systems of Lur'e type. While a more general viewpoint is adopted in \S \ref{sect:expcont}, the examples we consider are all of the form
\begin{equation}
\begin{aligned} \label{eq:lureproblem}
\dot z = Az + B\phi(Cz),
\end{aligned}
\end{equation}
where $A,B,C$ are matrices with values in $\KK = \CC$ or $\RR$, and $\phi$ is a continuous function. 

 In general, an autonomous system $\dot z = f(z)$ with $f$ continuous is said to be $r$-exponentially stable at a given rate $r \geq 0$ if $x=0$ is an equilibrium point and there is an increasing continuous function $\kappa$, with $\kappa(0) = 0$, such that 
\[
|z(t)| \leq \kappa(|z(0)|) \exp(-rt)
\]
for each solution $x$ and every $t \geq 0$ for which $x$ exists (a posteriori these solutions are bounded, hence exist globally).

Criteria for exponential stability of systems \eqref{eq:lureproblem}, including frequency domain formulations, can be deduced from standard results in hyperstability and dissipativity theory. A brief but self-contained exposition of the necessary material is presented in \S \ref{sect:expcont}, complete with simple proofs.
Special attention is given to critical cases where the linear block has marginal stability properties, is not necessarily minimal, and the frequency inequalities are non-strict.

 In particular, we provide a simplified proof of Popov's hyperstability theorem \cite[\S 16, Theorem 1]{popov1973hyperstability}. This result (in its full generality) is often overlooked in modern literature on the subject, but deserves to be better known. We also provide a new version of Popov's criterion for exponential stability that applies to nonlinearities with uncontrolled growth.

Although frequency criteria for exponential stability have been studied for a long time (see, e.g., \cite{bliman2002absolute,boczar2015exponential,hu2016exponential,accikmecse2008stability} for some recent results), even the simplest applications to mechanical systems are practically absent in the literature. A class of examples related to the Li\'enard equation is discussed in \S \ref{subsect:applicationscontinuous} below. For many of these examples, a direct time-domain approach is impractical; in contrast, the frequency-domain criteria described in this paper provide a systematic reduction in the number of free parameters, which substantially simplifies the analysis. These examples show that frequency methods still have an important role to play in analytic proofs of stability results for simple mechanical systems. 

\subsection{Applications} \label{subsect:applicationscontinuous} 
The simplest applications are to second-order equations. As an illustration we consider a class of dissipative Hamiltonian systems on $\RR^{2d}$ of the form
\begin{equation} \label{eq:hamiltonianformulation}
\dot z = (J-S)\nabla H(z),
\end{equation}
where $J$ is the usual symplectic matrix, and the symmetric part $S \geq0$ represents resistive elements of the system. Here 
\[
z = (q,p) \in \RR^{2d}, \quad  H(x) = \tfrac{1}{2}|p|^2 + f(q),
\]
and $f : \RR^d \rightarrow \RR$ is a given potential.  Assume that $S = \diag(\tau I,\,2\sigma I)$, where $\sigma > 0$ and $\tau \geq 0$ are constant, so \eqref{eq:hamiltonianformulation} becomes
\begin{equation}
\tag{H}
\begin{aligned} \label{eq:H}
\dot q &= p - \tau \nabla f(q), \\
\dot p &= -2\sigma p - \nabla f(q).
\end{aligned}
\end{equation}
The system \eqref{eq:H} with $\tau = 0$ represents a class of Rayleigh--Li\'enard oscillators with constant damping:
\begin{equation}
\tag{H$_0$}
\begin{aligned} \label{eq:H0}
\dot q &= p  \\
\dot p &= -2\sigma p - \nabla f(q).
\end{aligned}
\end{equation}
The case $\tau \geq 0$ has received less attention. Note that \eqref{eq:H} is equivalent to the second-order equation
\begin{equation} \label{eq:hessiandamped}
\ddot q + \left( 2\sigma + \tau^2 \nabla^2 f(q)\right) \dot q + (1+2\sigma \tau \nabla f(q)) = 0.
\end{equation}
 Equations of this form are sometimes called Hessian-damped; although several works have considered \eqref{eq:hessiandamped}, including other first-order representations (for example \cite{alvarez2002second,attouch2018fast,attouch2016fast,shi2018understanding,attouch2002optimizing,attouch2014dynamical,attouch2019first}), its simple Hamiltonian formulation appears to have been overlooked.  Discrete analogues of \eqref{eq:H} with $\tau > 0$ play an important role in accelerating the convergence of iterative convex optimization methods \cite{shi2018understanding}.
 
 To state the main results, we recall some classes of nonlinearities. Given $m < L \leq \infty$, a continuous function $\phi :   \KK^n \rightarrow \KK^n$
 is said to belong to the sector $[m,L]$ if 
 \begin{equation} \label{eq:sectorbounded1}
 \Re \langle mw - \phi(w), w - L^{-1}\phi(w) \rangle \leq 0
 \end{equation}
 for all $w \in \KK^n$, with the obvious interpretation when $L = \infty$.  Sector bounded nonlinearities play a central role in the classical passivity-based conditions for Lyapunov stability (e.g., the circle and Popov criteria --- cf. \S \ref{subsect:expstability}).
 
 Now let $\KK = \RR$, and suppose that $\phi = \nabla f$ for some potential $f : \RR^n \rightarrow \RR$. Even if $\phi$ belongs to a sector $[m,\infty]$, its potential can be  far from convex. Following \cite{necoara2019linear}, a $\mathcal{C}^1$ function $f : \RR^n \rightarrow \RR$ is said to be $m$-quasi-strongly convex (with respect to the origin) if $\nabla f$ belongs to the sector $[m,\infty]$ and
  \begin{equation}\label{eq:quasiconvex} 
 f(w) - \langle \nabla f(w),w \rangle + (m/2) |w|^2 \leq f(0)
 \end{equation}
for all $y \in \RR^n$. The latter condition further limits the possible concavity of $f$ versus its growth. This terminology is not entirely standard, and similar conditions appear in the literature under different names. 
\begin{theo} \label{theo:H0quasi}
	Suppose that $m >0,\, \sigma>0$, and $r >0$ satisfy
	\begin{equation} \label{eq:H0quasiparametercontraints}
	r \leq 2\sigma/3, \quad m \geq 2r\sigma -r^2.
	\end{equation}
	If $\nabla f$ belongs to the sector $[m,\infty]$ and $f$ is $m$-quasi-strongly-convex, then \eqref{eq:H0} is $r$-exponentially stable.
\end{theo}
Theorem \ref{theo:H0quasi} fails to show that the nonlinear problem is stable at the same exponential rate as the linear problem, namely the one where $f$ is quadratic and $\nabla^2 f \geq mI$. A spectral analysis (see Lemma \ref{lemm:tau=0linearconstraint}) shows that the linear problem is $r$-exponentially stable if and only if $r \leq \sigma$ and $m \geq 2r\sigma-r^2$, and at least one of these inequalities is strict.

The proof of Theorem \ref{theo:H0quasi} follows from a new Popov-type frequency criterion (see Lemma \ref{lemm:popov}). When at least one of the inequalities in \eqref{eq:H0quasiparametercontraints} is strict, feasibility of the frequency-domain criterion is equivalent to the existence of a time-domain Lyapunov function
\begin{equation} \label{eq:quadraticpluspotential}
V(z) = \langle Pz,z\rangle + f(q) - f(0)
\end{equation}
satisfying $\dot V \leq -2r V$. Sufficient conditions for the existence of a suitable matrix $P$ can be formulated as a semidefinite program, and an explicit formula is easily guessed from numerical experiments. One possibility is
\[
P = \begin{bmatrix}
r^2I & rI \\
rI & I/2
\end{bmatrix},
\]
which has the property that $P + \diag(mI/2,\,0)$ is positive definite precisely when at least one of the inequalities in \eqref{eq:H0quasiparametercontraints} is strict. Note that this $P$ depends on the rate $r$, and it is unclear how its form could be reasoned a priori.

When the upper sector bound is unknown, namely $\nabla f$ is only known to belong to the sector $[m,\infty]$, the best rate guaranteed by Theorem \ref{theo:H0quasi} occurs when 
\[
r = 2\sigma/3 - \varepsilon, \quad \sigma^2 = 9m/8,
\] 
yielding $r = \sqrt{m/2} - \varepsilon$ for any $\varepsilon>0$. This is in contrast to the linear problem over the same sector, where the best rate is 
\[
r = \sqrt{m} - \varepsilon
\]
for any $\varepsilon > 0$,
achieved in the critically damped regime $\sigma^2 = m$. If $r > 2\sigma/3$, then $r$-exponential stability still holds provided $\nabla f$ is known to lie in a certain finite sector:

\begin{theo} \label{theo:H0quasifinitesector}
	Suppose that $m >0,\, \sigma>0, \, r >0$, and $L> m$ satisfy
	\begin{gather*}
	2\sigma/3 < r < \sigma, \quad m \geq  8r^2 - 10r\sigma + 4 \sigma^2, \\ L \leq m + \frac{2(\sigma-r)}{3r - 2\sigma} \cdot \max\left( m-4r\sigma +4 r^2, 2m -4r\sigma + 2r^2\right).
	\end{gather*}
	If $\nabla f$ belongs to the sector $[m,L]$ and $f$ is $m$-quasi-strongly-convex, then \eqref{eq:H0} is $r$-exponentially stable.
\end{theo}

 If $\nabla f$ belongs to a finite sector (but $f$ is not quasi-strongly convex), a more conservative Popov-type criterion establishes the following weaker result.
\begin{theo} \label{theo:H0finitesectortimeinvariant}
	Suppose that $m >0,\, \sigma>0,\, r >0$, and $L> m$ satisfy
	\[
	r < \sigma, \quad m \geq 2r\sigma, \quad L \leq  \frac{(2\sigma-r)m}{r}.
	\]
	If $\nabla f$ belongs to the sector $[m,L]$, then \eqref{eq:H0} is $r$-exponentially stable.
\end{theo}

We can also consider the case where $\nabla f$ is time-dependent. Of course some mild conditions should be imposed to ensure that \eqref{eq:lureproblem} is solvable, but these are immaterial for the analysis below --- it certainly suffices to assume that $\nabla f$ is jointly continuous in $(t,z)$. Now suppose that $z \mapsto \nabla f(t,z)$ belongs to a sector $[m,L]$ uniformly in $t\geq 0$. A version of the classical circle criterion, which in the time domain corresponds to a quadratic Lyapunov function 
$V(z) = \langle Pz,z\rangle$, 
yields the following.

\begin{theo} \label{theo:H0finitesector}
	Suppose that $m>0,\,\sigma>0,\, r > 0$, and $L>m$ satisfy 
	\[
	r < \sigma, \quad m \geq 2r\sigma - r^2, \quad L \leq  m + 4(\sigma -r)(  \sigma - r +  (m-2r\sigma + r^2)^{1/2} ).
	\] 
	If $x\mapsto \nabla f(t,z)$ belongs to the sector $[m,L]$ for all $t\geq 0$, then \eqref{eq:H0} is $r$-exponentially stable.
\end{theo}

Theorem \ref{theo:H0finitesector} complements Theorem \ref{theo:H0quasifinitesector}  when $2r\sigma-r^2 \leq m < 8r^2 - 10r\sigma + 4 \sigma^2$, and complements Theorem \ref{theo:H0finitesectortimeinvariant} in the range  $2r\sigma-r^2 \leq m < 2r\sigma$.

When $\tau$ is allowed to be positive, a quantitative version of the Aizerman conjecture holds for \eqref{eq:H}: if $m,\,\sigma>0$ are fixed, then the optimal rate $r_\star$ that holds uniformly for linear $\nabla f$ in the sector $[m,\infty]$ as $\tau$ ranges over $[0,\infty)$ also holds for the nonlinear problem, with the same optimal value of $\tau$. This is true even if $\nabla f$ has time-dependence.

\begin{theo} \label{theo:tau>0}
Given $m>0$ and $\sigma > 0$, let
\[
r_\star=\frac{3\sigma + \sqrt{2 m+\sigma^2}}{2}, \quad \tau_\star = \frac{\sigma + \sqrt{2 m+\sigma^2}}{m}.
\]
If $x \mapsto \nabla f(t,z)$ belongs to the sector $[m,\infty]$ for all $t \geq 0$, then \eqref{eq:H} is $r_\star$-exponentially stable.
\end{theo}
As mentioned above, this result  is sharp in the following sense: if $r > r_\star$, then for each $\tau \geq 0$ there exists $k \in [m/2,\infty)$ such that the conclusion of  Theorem \ref{theo:tau>0} is false when $f(q)  = k|q|^2$. For simplicity, do not consider the finite sector analogues of Theorem \ref{theo:tau>0}.

All of the results in this section are established by verifying that certain frequency criteria hold, parameterized by a number of Lagrange multipliers that must be chosen appropriately. We stress that the results above all represent optimal applications of these criteria with respect to the multipliers.
As far as we are aware, Theorems \ref{theo:H0quasi}, \ref{theo:H0quasifinitesector}, \ref{theo:H0finitesectortimeinvariant}, and \ref{theo:H0finitesector} provide the best known global convergence rates for \eqref{eq:H0} under the given hypotheses. Theorem \ref{theo:tau>0}, which quantifies the stabilizing role played $\tau > 0$, also appears to be new; weaker results appear in \cite{shi2018understanding,attouch2019first}.

\subsection{Notation}
\begin{inparaenum}
		\item Let $A(s)$ be a matrix with entries in the field $\KK(s)$ of rational functions. Denote by $\nrank A(s)$ its \emph{normal rank}, namely its rank over $\KK(s)$. Equivalently, 
	\[
	\nrank A(s) = \max\{ \rank A(s_0): s_0 \in  \CC \text{ is not a pole of } A(s)\}.
	\]
	and $\nrank A(s) = \rank A(s_0)$ for all but finitely many $s_0 \in \CC$. 
	
	\item Given a square matrix $M$, we write $\Re M = (M+ M^*)/2$.

\end{inparaenum}

\section{Exponential stability} \label{sect:expcont}

\subsection{Dissipativity}
A detailed treatment of dissipativity theory can be found in the foundational papers \cite{willems1972dissipativeI,willems1972dissipativeII}, as well as  \cite{hill1976stability,hill1980dissipative}. In this section we provide a brief review of the relevant material.  A parallel development for nonlinear systems is given in \cite{willems1972dissipativeI,hill1976stability}.
Consider a linear time-invariant (LTI) system with state $x \in \KK^m$, input $u \in \KK^n$:
\begin{equation} \label{eq:LTI}
\begin{aligned}
\dot x &= Ax + Bu.
\end{aligned}
\end{equation}
Unless otherwise specified, we restrict our attention to inputs $u \in L^2_\loc(\RR_+; \KK^n)$ which are locally square-integrable. Dissipativity is defined with respect to a fixed quadratic form $\sigma : \KK^{m+n} \rightarrow \RR$,
referred to as the supply rate. Write
 \begin{equation} \label{eq:generalsupply}
\sigma(x,u) = \langle Q x,x \rangle + 2\Re \langle Sx,u \rangle + \langle Ru,u \rangle,
\end{equation}
where $Q,\,S,\,R$ are matrices of the appropriate sizes with $Q =Q^*$ and $R=R^*$. For $[t_0,t_1] \subset [0,\infty]$,
the integral
\begin{equation} \label{eq:energy}
E(x,u,t_0,t_1) = \int_{t_0}^{t_1} \sigma(x(s),u(s)) \, ds
\end{equation}
is interpreted as the energy supplied to the system over the time period $[t_0,t_1]$ when driven by the input $u$. Finally, let 
\begin{equation} \label{eq:Mdef}
M = \begin{bmatrix}
Q &S \\
S^* & R
\end{bmatrix}
\end{equation}
denote the Hermitian matrix associated with $\sigma$.

\begin{defi} \label{defi:cyclodissipative}
A function $V : \KK^m \rightarrow \RR$ is called a \emph{storage function} for $(A,B,M)$ if the dissipation inequality
\begin{equation} \label{eq:virtualstorage}
V(x(t_1)) - V(x(t_0)) \leq E(x,u,t_0,t_1)
\end{equation}
holds 
whenever $[t_0,t_1]\subset [0,\infty)$ and $\dot x = Ax + Bu$. The triple $(A,B,M)$ is said to be \emph{cyclodissipative} if it admits a storage function.
\end{defi}

In many references, storage functions are required to be nonnegative.
There are at least two distinguished candidates for storage functions (when they exist). Define the available storage $\avail$ by
\[
\begin{gathered}
\avail(x_0) = \sup \{ -E(x,u,0,T): \dot x = Ax + Bu, \, T\geq0, \, x(0) = x_0, \, x(T) = 0\},
\end{gathered}
\]
which represents the greatest amount of energy that can be extracted in motions driving the system from state $x_0$ to the origin. Similarly, the required supply $\req$ is defined by
\[
\begin{gathered}
\req(x_0) = \inf  \{ E(x,u,0,T) : \dot x = Ax + Bu, \, T\geq0, \, x(0) = 0, \, x(T) = x_0 \}.
\end{gathered}
\]
If $(A,B)$ is controllable, then $\req < \infty$ and $\avail > -\infty$. The following result gives several characterizations of cyclodissipativity with respect to a given supply rate; the proof follows from straightforward manipulations of the definitions.
\begin{lemm}[\cite{willems1972dissipativeI,willems1972dissipativeII}] \label{lemm:cyclodissipative}
If  $(A,B)$ is controllable, then the following conditions are equivalent:
\begin{enumerate} \itemsep6pt
	\item  $(A,B,M)$ is cyclodissipative,
	\item $\req > -\infty$,
	\item $\avail < \infty$,
	\item  $E(x,u,0,T) \geq 0$ whenever $\dot x = Ax + Bu$ and $T \geq 0$ is such that $x(0) = x(T)$.
\end{enumerate}
If any of these conditions are satisfied, then $\req (0) = \avail(0) = 0$, and any storage function normalized by $V(0) = 0$ satisfies 
\[
 \avail \leq V \leq \req.
\]
\end{lemm}
Because the underlying system is linear, even more is true: there exist Hermitian matrices $P_- \leq P_+$ satisfying 
\begin{equation} \label{eq:maximumsolution}
\req(x) = \langle P_- x,x\rangle, \quad \avail = \langle P_+ x,x\rangle,
\end{equation}
and if $(A,B)$ is controllable, then $(A,B,M)$ is cyclodissipative if and only if it admits a quadratic storage function. When $\KK = \RR$, the matrices $P_\pm$ can be chosen with real entries.

\begin{defi} \label{defi:dissipative}
The triple $(A,B,M)$ is said to be \emph{dissipative} if it admits a \emph{nonnegative} storage function.
\end{defi}

The following analogue of Lemma \ref{lemm:cyclodissipative} holds.

\begin{lemm} [\cite{willems1972dissipativeI,willems1972dissipativeII}] \label{lemm:dissipative}
If $(A,B)$ is controllable, then the following conditions are equivalent:
\begin{enumerate} \itemsep6pt
	\item  $(A,B,M)$ is dissipative,
	\item $\req  \geq 0$,
	\item $E(x,u,0,T) \geq 0$ for every $T\geq 0$ whenever $\dot x = Ax + Bu$ and $x(0) = 0$.
\end{enumerate}
\end{lemm}

From the differential version of the dissipation inequality \eqref{eq:virtualstorage}, cyclodissipativity is equivalent to the existence of $P = P^*\in \KK^{m \times m}$ satisfying the linear matrix inequality 
\begin{equation} \label{eq:LMI} \tag{LMI}
\Lambda(P) = 
\begin{bmatrix}
2\Re A^*P & PB \\
B^*P & 0  
\end{bmatrix} - \begin{bmatrix}
Q &S^* \\
S & R
\end{bmatrix} \leq 0.
\end{equation} 
If $P$ can be chosen positive semidefinite, then $(A,B,M)$ is dissipative.

\subsection{KYP Lemma}
First we recall necessary and sufficient frequency conditions under which the linear matrix inequality \eqref{eq:LMI} admits a Hermitian solution $P$.
 Introduce the Popov function
 \[
 \Pi(\eta, \zeta) = 	\begin{bmatrix}
 (\bar \eta I - A)^{-1}B \\
 I
 \end{bmatrix}^* \begin{bmatrix}
 Q &  S^* \\
 S &  R
 \end{bmatrix} 	\begin{bmatrix}
 (\zeta I- A)^{-1}B \\
 I
 \end{bmatrix},
 \]
 which is a meromorphic function of $(\eta, \zeta)$. If $P \in \KK^{m\times m}$, then
 \begin{multline} \label{eq:popovpositive}
 \begin{bmatrix}
 (\bar \eta I - A)^{-1}B \\
 I
 \end{bmatrix}^* \begin{bmatrix}
 2\Re A^*P & PB \\
 B^*P & 0  
 \end{bmatrix} \begin{bmatrix}
 (\zeta I - A)^{-1}B \\
 I
 \end{bmatrix} \\ =(\zeta+ \eta)B^*(\eta I-A^*)^{-1}P (\zeta I-A)^{-1}B.
 \end{multline}
If $\Lambda(P) \leq 0$ admits a Hermitian solution, then it follows from \eqref{eq:popovpositive} and \eqref{eq:LMI} that the frequency condition 
 \begin{equation} \label{eq:KYPFDI} \tag{FDI}
 \Pi(-i\omega, i\omega) \geq 0 \text{ whenever } \omega \in \RR \text{ and } \det(i\omega I -A)\neq 0
 \end{equation}
 holds. Although the following version of the Kalman--Yakubovich--Popov (KYP) lemma is well-known, a short proof is included for the reader's convenience.
 
 \begin{lemm} \label{lemm:KYP}
If $(A,B)$ is controllable, then $\Lambda(P) \leq 0$ admits a solution $P=P^*$ if and only if \eqref{eq:KYPFDI} holds.
 \end{lemm}
\begin{proof}
According to Lemma \ref{lemm:cyclodissipative}, to show that $\Lambda(P) \leq 0$ admits a solution (namely $(A,B,M)$ is cyclodissipative), it suffices to show that $\avail < \infty$, where the supply rate $\sigma$ is given by \eqref{eq:generalsupply}. Let $T \geq 0$. Suppose that $w = (x,u)$ satisfies
\begin{equation} \label{eq:admissiblexu}
\dot x = Ax + Bu, \quad x(0) = x^0, \quad x(T) = 0.
\end{equation}
Since $(A,B)$ is controllable, there exists $w_1 = (u_1,x_1)$ of compact support satisfying \eqref{eq:admissiblexu}, such that
\[
\supp w_1 \subset (-\infty,T], \quad w(t) = w_1(t) \text{ for } t\in [0,T].
\]
Moreover, the restriction of $w_1$ to $(-\infty,0)$ can be chosen to depend only on $x^0$. By Plancherel's formula, \eqref{eq:KYPFDI} implies that $E(x_1,u_1,-\infty,T) \geq 0$. In particular,
\[
E(x_1,u_1,-\infty,0) \geq -E(x,u,0,T),
\]
which shows that $\avail(x^0) < \infty$ for all $x^0 \in \KK^m$.
\end{proof}

Given a solution $P = P^*$ of $\Lambda(P) \leq 0$, one can always find $K \in \KK^{q \times m}$ and $L \in \KK^{q \times n}$ such that
\begin{equation} \label{eq:lure}
\begin{bmatrix}
2\Re A^*P & PB \\
B^*P & 0  
\end{bmatrix} - \begin{bmatrix}
{Q} &  {S}^* \\
{S} &  {R}
\end{bmatrix} =  -  \begin{bmatrix}
K^* \\ L^*
\end{bmatrix} \begin{bmatrix}
K & L
\end{bmatrix},
\quad \rank \begin{bmatrix}
K & L
\end{bmatrix} = q.
\end{equation}
A priori, the number of rows $q$ is bounded above by $m+n$. Referring to \eqref{eq:popovpositive}, any such decomposition allows us to write
\begin{equation} \label{eq:popovfactorization}
\Pi(\eta, \zeta) = (\zeta+ \eta)B^*(\eta I-A^*)^{-1}P (\zeta I-A)^{-1}B + G(\bar \eta)^*G(\zeta),
\end{equation}
where the transfer matrix $G(s) = L + K(sI-A)^{-1}B$ has $q$ rows. This in turn provides a spectral factorization
\begin{equation} \label{eq:spectralfactorization}
\Pi(-s,s) = G(-\bar s)^*G(s).
\end{equation}
Restricting \eqref{eq:spectralfactorization} to the imaginary axis shows that $q \geq \nrank \Pi(-s,s)$.

Recall the notation $P_+$ and $P_-$ for the Hermitian matrices corresponding to the required supply and available storage, respectively (when they exist). If $(A,B)$ is controllable and $(A,B,M)$ is dissipative, then
\begin{equation} \label{eq:rankattained}
q = \nrank \Pi(-s,s)
\end{equation}
for any factorization \eqref{eq:lure} associated with $P_\pm$. It is also well-known that the transfer matrix corresponding $G(s)$ to an extremal solution has certain stability properties. Since these facts will be used later on, we present an elementary proof adapted from \cite{megretski2010kyp}  (when $R$ is nonsingular the result follows from classical results about the algebraic Riccati equation \cite{willems1971least}).

\begin{lemm} \label{lemm:lure} Let $(A,B)$ be controllable, and suppose that $\Lambda(P) \leq 0$ admits a Hermitian solution. If $(K_\pm, L_\pm)$ is such that the factorization \eqref{eq:lure} holds for $P_\pm$, then
	\[
	\dim \ker \begin{bmatrix}A-sI & B \\
	K_\pm & L_\pm
	\end{bmatrix}^* = 0 \text{ whenever } \pm \Re s < 0.
	\]
\end{lemm}
\begin{proof}
	 We only consider $P_+$, with the $P_-$ case being similar. Suppose on the contrary that there exists $\lambda \in \CC$ such that $\Re \lambda < 0$ and 
	\begin{equation} \label{eq:lurekernel}
	\begin{bmatrix}
	v \\ z
	\end{bmatrix}^*\begin{bmatrix}A-\lambda I & B \\
	K_+& L_+
	\end{bmatrix} = 0, \quad \begin{bmatrix}
	v \\ z
	\end{bmatrix} \neq 0.
	\end{equation}
	Note that $v \neq 0$, otherwise $z = 0$ as well since $\begin{bmatrix}
	K_+ & L_+
	\end{bmatrix}$ has full row rank. Also $z \neq 0$ since $(A,B)$ is controllable. If $\dot x = Ax + Bu$, then from \eqref{eq:lurekernel} the function $e^{-\lambda t}\langle x,v\rangle$ satisfies
	\[
	e^{-\lambda T} \langle x(T),v \rangle -\langle x(0), v\rangle= -\int_{0}^T e^{-\lambda s} \, \langle K_+ x(s) + L_+u(s),z  \rangle \, ds.
	\]
	Suppose that $x(0) = 0$ and $x(T) = x^0$. Since $z \neq 0$ and $\Re \lambda < 0$, by Cauchy--Schwarz there exists $c > 0$ such that
	\[
	c|\langle x^0, v\rangle|^2 \leq \int_{0}^T |K_+ x(t) + L_+ u (t)|^2 \, dt.
	\]
	Crucially, this $c$ is independent of $T$ as well. From the dissipation inequality,
	\[
	\req(x_0)+ \int_{0}^T  |K_+ x(t) + L_+ u (t)|^2 \, dt = E(x,u,0,T).
	\]
	Now choose $x^0 \neq 0$ such that $\langle x_0, v\rangle  \neq 0$. Taking the infimum over all $T \geq 0$ and $(u,x)$ with $x(0) = 0$ and $x(T) = x^0$ contradicts the definition of $\req$.
\end{proof} 
We now show that \eqref{eq:rankattained} holds. For any rational function $G(s) = L + K(sI-A)^{-1}B$, there holds the identity
\begin{align*}
\nrank \begin{bmatrix}
A - sI & B \\
K & L
\end{bmatrix}  &= \nrank \begin{bmatrix}
I & 0 \\
K(sI - A)^{-1} & I
\end{bmatrix}   \begin{bmatrix}
A - sI & B \\
K & L
\end{bmatrix}  \\ 
&= \nrank \begin{bmatrix}
A - sI & B \\
0 & G(s)
\end{bmatrix} = 
m + \nrank G(s).
\end{align*}
In the setting of Lemma \ref{lemm:lure}, the normal rank of the left-hand side is $m+q$, which indeed implies \eqref{eq:rankattained}. If $\Pi(-s,s)$ has full normal rank, then $L$ is necessarily a square matrix.

To conclude this section, consider the behavior of \eqref{eq:LMI} and the Popov function under the transformations
\begin{equation} \label{eq:statefeedback}
x_\sim = T^{-1}x, \quad u_\sim = u - Fx,
\end{equation}
where $T \in \KK^{m \times m}$ is invertible and $F \in \KK^{n\times m}$. Given $(A,B,M, P)$, let
\begin{equation} \label{eq:feedbackcontrol}
\begin{aligned}
&A_\sim =  T^{-1}(A + BF)T, \\ &B_\sim = T^{-1}B, \\
 &Q_\sim = T^*(Q + 2\Re S^*F + F^* RF)T, \\ &R_\sim = R, \\ &S_\sim = (S + RF)T, \\
 &P_\sim = T^*PT,
\end{aligned}
\end{equation}
where $M_\sim$ is defined as in \eqref{eq:Mdef} by replacing $(Q,R,S)$ with $(Q_\sim, R_\sim, S_\sim)$.
If we define the affine matrix function $\Lambda_\sim$  as in \eqref{eq:LMI} by replacing $(A,B,M)$ with $(A_\sim, B_\sim, M_\sim)$, then
\[
\Lambda_\sim(P_\sim) = \begin{bmatrix}
T & 0 \\
F & I
\end{bmatrix}^* \Lambda(P) \begin{bmatrix}
T & 0 \\
F& I
\end{bmatrix},
\]
and $E_\sim(x_\sim, u_\sim, 0, t) = E(x,u,0,t)$ when the left-hand side is defined in the obvious way. In particular, dissipativity of $(A,B,M)$ is preserved under changes of variable of the form \eqref{eq:statefeedback}. Furthermore, if $\Pi_\sim(\eta,\zeta)$ is the transformed Popov function, then
\[
 \Pi(\eta,\zeta) = W(\bar \eta)^* \,  \Pi_\sim(\eta,\zeta)\,W(\zeta),
\]
where  
\[
W(s) = I - F(sI-A)^{-1}B.
\]
 In particular, $\Pi(\bar s, s) \geq 0$ if and only if $\Pi_\sim(\bar s, s) \geq 0$, apart from finitely many $s$. It is also clear that if $(A,B)$ is controllable, then so is $(A_\sim, B_\sim)$.

\subsection{Nonnegative storage functions} \label{subsect:semidefinite}

In general it is difficult to characterize the existence of positive semidefinite solutions to $\Lambda(P) \leq 0$. The following notion of minimal stability is due to Yakubovich and Popov.

\begin{defi} \label{defi:minimal}
$(A,B,M)$ is \emph{minimally stable} if for each $x^0 \in \KK^m$ there exists $(x,u)$ with $\dot x = Ax + Bu$ such that
\[
 x(0)=x^0, \quad E(x,u,0,t) \leq 0 \text{ for all } t\geq 0, \quad x(t) \rightarrow 0 \text{ as }t \rightarrow \infty.
\]
\end{defi}
If $(A,B,M)$ is minimally stable, then every continuous storage function normalized by $V(0) = 0$ is nonnegative: for each $x^0 \in \KK^m$, minimal stability and the dissipation inequality imply that
\[
V(x(t)) \leq V(x^0),
\]
at which point it suffices to let $t \rightarrow \infty$ to deduce that $V(x^0) \geq 0$.  In particular, every Hermitian solution of $\Lambda(P) \leq 0$ is positive semidefinite. Moreover, it is easy to see that minimal stability is invariant under the transformations considered in \eqref{eq:feedbackcontrol}. In the notation there, $(A,B,M)$ is minimally stable if and only if $(A_\sim, B_\sim, M_\sim)$ is.

 If $A$ is Hurwitz and $ Q \leq 0$, then $(A,B,M)$ is minimally stable: simply take $u = 0$ and let $x$ solve the asymptotically stable system $\dot x = Ax$ with $x(0) = x^0$. Equivalently, this can be seen from the Lyapunov inequality
\[
2\Re A^*P \leq Q
\]
implied by the upper left block of \eqref{eq:LMI}. More generally, invariance under state feedback implies the following.
\begin{lemm} \label{lemm:minimal}
	If there exists $F \in \KK^{n \times m}$ such that $A + BF$ is Hurwitz and 
	\begin{equation} \label{eq:minimalnegative}
	Q+ 2\Re S^*F + F^* RF \leq 0,
	\end{equation}
	then $(A,B)$ is minimally stable.
\end{lemm}
\begin{proof}
	Apply \eqref{eq:feedbackcontrol}.
\end{proof}

If $Q \leq 0$, then \eqref{eq:minimalnegative} holds for $F = -\delta S$ provided $\delta \in [0,2/\|R\|]$. Thus, one way to verify the hypotheses of Lemma \ref{lemm:minimal} is to  show that $A -\delta BS$ is Hurwitz for $\delta$ in this range.

Another path towards positivity involves strengthened frequency domain conditions. From \eqref{eq:popovpositive}, if $\Lambda(P) \leq 0$ admits a solution $P \geq 0$, then 
\begin{equation} \label{eq:popovnegative} \tag{FDI$_+$}
\Pi(\bar s, s) \geq 0 \text{ whenever } \Re(s) \geq 0 \text{ and } \det(sI-A)\neq 0.
\end{equation}
This is of course a stronger condition than \eqref{eq:KYPFDI} in general.
It was claimed in \cite{willems1971least} that the converse also holds, but Willems later provided a counterexample \cite{willems1974existence}. In the latter reference it was pointed out that sufficiency does hold if there is a decomposition
\begin{equation} \label{eq:differenceofsquares}
\sigma(x,u) = |C_1 x + D_1 u|^2 -|C_2 x + D_2 u |^2,
\end{equation}
where $D_1$ is square; for a complete proof, see \cite{reis2015kalman}. The following closely related condition was essentially considered by Moylan \cite{moylan1975frequency} (although some details are missing in the latter reference).

\begin{lemm} \label{lemm:moylan}
	If $(A,B)$ is controllable and there exists $F \in \KK^{n\times m}$ such that
	\begin{equation} \label{eq:moylan}
	Q+ 2\Re S^*F + F^*RF \leq 0,
	\end{equation}
	then $\Lambda(P) \leq 0$ admits a solution $P \geq 0$ if and only if \eqref{eq:popovnegative} holds.
\end{lemm}
\begin{proof}
We show that \eqref{eq:popovnegative} implies the existence of  a solution $P \geq 0$ to $\Lambda(P) \leq 0$. From \eqref{eq:moylan} and \eqref{eq:feedbackcontrol}, we can assume that $Q \leq 0$. First, suppose that $R > 0$. With the feedback $F = -R^{-1}S$, we can furthermore assume our system $(A,B,M)$ satisfies 
\[
M = \begin{bmatrix}
Q - SR^{-1}S & 0 \\
0 & R
\end{bmatrix}.
\]
If we factor $Q - SR^{-1}S = -C^*C$ for some $C$, then $\Pi(\bar s, s) = -H(s)^*H(s) + R$, where $H(s) = C(sI-A)^{-1}B$. If $C = 0$, then $P = 0$ solves $\Lambda(P) \leq 0$. Otherwise, in an appropriate basis, we can assume that $A$ has a detectability decomposition
\begin{equation} \label{eq:kalman}
A = 
\begin{bmatrix}
A_{11} & 0 \\
A_{21} & A_{22}
\end{bmatrix}, \quad B = \begin{bmatrix}
B_{1} \\ B_{2}
\end{bmatrix}, \quad C = \begin{bmatrix}
C_{1} & 0
\end{bmatrix},
\end{equation}
where $(A_{11},C_1)$ is detectable and $(A_{11},B_1)$ is controllable. Now \eqref{eq:popovnegative} implies that $H(s)$ has no poles in the closed right half-plane; since the poles of $H(s)$ and the eigenvalues of $A_{11}$ coincide in the closed right half-plane, we conclude that $A_{11}$ is Hurwitz. Since \eqref{eq:KYPFDI} holds, there is a Hermitian solution 
\[
P = \begin{bmatrix}
P_{11} & P_{12} \\
P_{21} & P_{22}
\end{bmatrix}
\]
to $\Lambda(P) \leq 0$, where the block decomposition is the same as in \eqref{eq:kalman}.  Thus $P_{11} \geq 0$, since it satisfies the Lyapunov inequality
\[
 2\Re A_{11}^*P_{11}  \leq -C_1^*C_1.
\]
 It is easy to check that if
 \[
 P_0 = \begin{bmatrix}
 P_{11} & 0 \\
 0 & 0\end{bmatrix},
 \] 
 then $P_0 \geq 0$ and $\Lambda(P_0) \leq 0$ as well.
 
 Even if $R$ is not invertible, the argument above applies if $R$ is replaced by $R+\varepsilon I$ for any $\varepsilon > 0$. This implies the dissipation inequality
 \[
 \int_{0}^T \sigma(x(s),u(s)) + \varepsilon|u(s)|^2 \, ds  \geq 0
 \]
 for every $T \geq 0$ whenever $\dot x= Ax + Bu$ and $x(0) = 0$. Let $\varepsilon \rightarrow 0$ and apply Lemma \ref{lemm:dissipative} to finish the proof.
\end{proof}

In general, the existence of a decomposition \eqref{eq:differenceofsquares} is distinct from the hypotheses of Lemma \ref{lemm:moylan}. However, when $R> 0$, they coincide; this is easily deduced from the fact that \eqref{eq:differenceofsquares} implies $M$ has at most $n$ positive eigenvalues.

The relationship is between Lemmas \ref{lemm:minimal} and \ref{lemm:moylan} may not be clear at first glance. To complete our discussion, we show that they are equivalent when $(A,S)$ is detectable. For simplicity we assume that $Q \leq 0$.

\begin{lemm}
	Suppose that $(A,B)$ is controllable, $(A,S)$ is detectable, and $Q \leq 0$. If $\eqref{eq:KYPFDI}$ holds, then the following are equivalent.
	\begin{enumerate} \itemsep6pt
		\item  $(A,B)$ is minimally stable, \label{it:hurwitz1}
		\item Every Hermitian solution of $\Lambda(P) \leq 0$ satisfies $P \geq 0$, \label{it:hurwitz2}
		\item \eqref{eq:popovnegative} holds, \label{it:hurwitz3} 
		\item There exists $\delta \in (0,2/\| R\|)$ such that $A - \delta BS$ is Hurwitz. \label{it:hurwitz4}
	\end{enumerate}
\end{lemm}
\begin{proof}
Clearly $\eqref{it:hurwitz1} \Rightarrow \eqref{it:hurwitz2} \Rightarrow \eqref{it:hurwitz3}$ and $\eqref{it:hurwitz4} \Rightarrow \eqref{it:hurwitz1}$, so it remains to show that $\eqref{it:hurwitz3} \Rightarrow \eqref{it:hurwitz4}$. If $G(s) = S(sI-A)^{-1}B$ and $G_\delta(s) = S(sI - A  + \delta BS)^{-1}B$, then
\[
G_\delta(s) = (I + \delta G(s))^{-1} G(s).
\]
Since $(A-\delta BS,B)$ is controllable and $(A-\delta BS,S)$ is detectable, every eigenvalue of $A - \delta BS$ in the closed right half-plane is a pole of $G_\delta(s)$. If $\delta > 0$, then $G_\delta(s)$ has a pole at $s \in \CC$ if and only if $s$ is not a pole of $G(s)$ and $\det (I + \delta G(s)) = 0$. But if $\Re s \geq 0$ and $s \in \CC$ is not a pole of $G(s)$, then
\[
2 \Re G(s) + R \geq \Pi(s) \geq 0
\]
since $Q \leq 0$. Thus $\det(I+\delta G(s)) \neq 0$ provided $\delta \in (0, 2/\| R \|)$,  and hence $A-\delta BS$ is Hurwitz (cf. \cite[Lemma 1]{willems1972dissipativeII}).
\end{proof}

 There exist other frequency conditions under which $\Lambda(P) \leq 0$ admits positive semidefinite solutions \cite{molinari1975conditions,trentelman2001pick,willems1998quadratic,popov1973hyperstability}, but these are not discussed here.

\subsection{Hyperstability}
Next, we turn our attention to hyperstability, which is a notion of robust stability for LTI systems due to Popov \cite{popov1973hyperstability}.
\begin{defi} $(A,B,M)$ is \emph{hyperstable} if there exists $c>0$ such that for all $t \geq 0$,
	\begin{equation} \label{eq:hyperstabilityineq}
	|x(t)| \leq c (|x(0)| + \beta) 
	\end{equation}
 whenever $\dot x = Ax + Bu$ and $E(x,u,0,t) \leq \beta^2$ for all $t \geq 0$.
\end{defi}
Hyperstability is closely related to Safonov's graph separation property and the modern theory of integral quadratic constraints \cite{megretski1997system,safonov1979stability,carrasco2018conditions,scherer2018stability}, but these connections are not pursued here.

\begin{lemm} \label{lemm:hyperstabilityimpliesFDI}
	If $(A,B,M)$ is hyperstable, then \eqref{eq:popovnegative} holds.
\end{lemm}
\begin{proof}
 Suppose that $s \in \CC$ is not an eigenvalue of $A$ and that $\Pi(\bar s, s)$ is not positive-semidefinite. Then there exists a nonzero $u^0 \in \KK^n$ such that 
	\[
	\langle \Pi(\bar s, s)u^0, u^0\rangle < 0.
	\]
	 If $u(t) = e^{s t}u^0$ and $x(t) = (sI-A)^{-1}u(t)$, then clearly $E(x,u, 0 ,t) \leq 0$ for all $t \geq 0$. On the other hand, if $\Re s >0$, then $|x(t)| \rightarrow \infty$ as $t \rightarrow \infty$, which contradicts \eqref{eq:hyperstabilityineq}. Thus $\Pi(\bar s, s) \geq 0$ for all $\Re s > 0$, and hence also for $\Re s = 0$ by continuity (away from the eigenvalues of $A$).
\end{proof}

Popov formulated sufficient conditions for the equivalence between \eqref{eq:KYPFDI} and hyperstability. Proofs and even precise statements of these results are difficult to find in the English language literature. Furthermore, Popov's original argumentation relies on the construction of a normal form for the output-zeroing problem whose complexity detracts from the underlying ideas; in this section we provide streamlined proofs.

We will  need some basic results on the existence of zero-dynamics for LTI systems. Consider an LTI system with input $u \in \KK^n$, state $x \in \KK^n$ and output $v \in \KK^q$:
\begin{align*}
\dot x &= Ax + Bu, \\
v &= Cx + Du.
\end{align*}
Denote by $\mathcal{Z}$ the set of all functions $(x,u)$ such that $\dot x = Ax + Bu$ and $Cx + Du =0$ a.e., where $u \in L^1_\loc(\RR_+; \KK^n)$ and $x$ is absolutely continuous. For a different perspective on the following lemma (whose constructive proof is close in spirit to \cite[\S 36]{popov1973hyperstability}), see \cite[\S 6.1]{isidori2013nonlinear}.  

\begin{lemm} \label{lemm:zerodynamics}
If $G(s) = D + C(sI-A)^{-1}B$ has full normal column rank, then there exists $F \in \KK^{n \times m}$ such that $(x,u) \in \mathcal{Z}$ implies $u = Fx$ a.e., and 
\[
\dot x = (A+BF)x.
\]
Furthermore, $\range D$ is orthogonal to $\range (C+DF)$.
\end{lemm}
\begin{proof}
	Since $\nrank G(s)= n$, the restriction of $G(s)$ to any subspace of $\RR^n$ is not identically zero. Of course the normal rank of $G(s)$ is also invariant under any feedback 
	\[
	(A, B, C, D) \mapsto (A+ BF, B, C+ DF, D).
	\]
	Temporarily assume that $D=0$. Set $\mathcal{U} = \RR^n$ and $\mathcal{V} = \RR^q$. Thus we view $B$ as a function on $\mathcal{U}$, and $C$ as a function with values in $\mathcal{V}$.  Since $G(s)$ is not identically zero, there exists a smallest integer $p \geq 0$ for which the Markov parameter $CA^p B$ is nonzero (otherwise $G(s)$ would vanish to infinite order at infinity). Define
	\[
	F_1 = -(CA^{p}B)^\dagger CA^{p+1},
	\] 
	where $(CA^pB)^\dagger$ denotes the Moore--Penrose inverse of the operator $CA^pB : \mathcal{U} \rightarrow \mathcal{V}$. Set $\mathcal{U}_1 = \ker CA^p B \subset \mathcal{U}$, and notice that $\dim U_1 < \dim U$ since $CA^p B \neq 0$. 
	
	Now suppose that $(x,u) \in \mathcal{Z}$. Differentiating  $p$ times the identity $Cx=0$ yields 
	\[
	CA^{p+1}x = -CA^p Bu
	\]
	a.e.
	Thus $u = F_1x + u_1$ a.e., where $u_1 = ((CA^pB)^\dagger (CA^pB) - I)u$ is contained  in $\mathcal{U}_1$, and
	\[
	\begin{aligned} \label{eq:reducedsystem}
	\dot x &= (A + BF_1) x + Bu_1, \\
	v &= Cx.
	\end{aligned}
	\]
	If $\mathcal{U}_1 = \{0\}$, then we are done. Otherwise, use that the transfer of the system \eqref{eq:reducedsystem} acting on $\mathcal{U}_1$ is not identically zero either. Repeat the same procedure, replacing $A$ with $A+ BF_1$. Since the dimension of the input space is strictly reduced at each step, we obtain a terminating chain of subspaces 
	\[
	\{ 0\} = \mathcal{U}_N \subset \cdots \subset \mathcal{U}_1 \subset \mathcal{U}
	\] 
	and a sequence $F_1, \ldots, F_N$. For $i > 1$ the range of $F_i$ is contained in $\mathcal{U}_{i-1} \subset \mathcal{U} = \RR^n$, and by a slight abuse of notation we consider $F_i$ as a map $\RR^m \rightarrow \RR^n$.  If  $F = F_1 + \cdots + F_N$, then iteratively following the steps above, we find that
	\[
	\dot x = (A+BF)x, \quad u = Fx
	\]
	whenever  $(x,u) \in \mathcal{Z}$, as desired.

Finally, suppose that $D \neq 0$. In that case, set $\mathcal{U}_0 = \ker D$ and $F_0 = -D^\dagger C$. Let $u_0 = u - F_0x$ and consider the transformed equation
\begin{align*} 
\dot x &= (A + BF_0)x + Bu_0, \\ 
v &= (C + DF_0)x + Du_0.
\end{align*}
If $(x,u) \in \mathcal{Z}$, then $u_0$ is entirely contained in $\mathcal{U}_0$. Now $I-D D^\dagger$ is the orthogonal projection onto $\mathcal{V}_0 = (\range D)^\perp$, so $\range D$ is orthogonal to $\range(C+DF_0)$ owing to
\[
C+DF_0 = (I-DD^\dagger)C.
\]
 To finish the proof, repeat the argument above from the beginning, replacing $(A,C)$ with $(A+BF_0, C+DF_0)$, the input space $\mathcal{U}$ with $\mathcal{U}_0$, and the output space $\mathcal{V}$ with $\mathcal{V}_0$. By construction $F_i$ for $i > 0$ takes values in $\mathcal{U}_0$, so 
 \[
 C+DF = C+DF_0
 \] 
 and hence $\range D$ remains orthogonal to $\range (C+DF)$.
\end{proof}

Lemma \ref{lemm:zerodynamics} implies that if $(x,u) \in \mathcal{Z}$, then $x$ is entirely contained in the unobservable subspace 
\[
\mathcal{S} = \bigcap_{l =0}^\infty \ker (C+DF)(A + BF)^l.
\]
In particular, the map $\mathcal{Z} \ni (x,u) \mapsto x(0)$ takes its values in $\mathcal{S}$. This map is an isomorphism  between  $\mathcal{Z}$ and $\mathcal{S}$, whose inverse is 
\begin{equation} \label{eq:zerodynamicsisomorphism}
\mathcal{S} \ni x^0 \mapsto (e^{(A+BF)t}x^0, Fe^{(A+BF)t}x^0).
\end{equation}
In the following lemma we continue to assume the hypotheses of Lemma \ref{lemm:zerodynamics}.

\begin{lemm} \label{lemm:zerodynamicsspectrum}
A number $s \in \CC$ is an eigenvalue of $(A+BF)|_{\mathcal{S}}$ if and only if
\[
\dim \ker \begin{bmatrix}
A- s I & B \\
C & D
\end{bmatrix} \neq 0.
\]
\end{lemm}
\begin{proof}
If $x^0$ is an eigenvector of $(A+BF)|_{\mathcal{S}}$ with eigenvalue $\lambda \in \CC$ and we define $u^0 = Fx^0$, then \eqref{eq:zerodynamicsisomorphism} implies 
\[
 \begin{bmatrix}
A- \lambda I & B \\
C & D
\end{bmatrix} \begin{bmatrix}
x^0 \\ u^0
\end{bmatrix} = 0.
\]
The converse argument is similar.
\end{proof}

In order to show that a given system is hyperstable, it is sufficient (but not necessary) to show that the system is dissipative with a positive definite quadratic storage function. This follows immediately from the dissipation inequality \eqref{eq:virtualstorage}.

\begin{prop}[{\cite[\S 16, Theorem 1]{popov1973hyperstability}}] \label{prop:popovhyperstability1}
Suppose $(A,B)$ is controllable, $(A,B,M)$ is minimally stable, and $\det \Pi(-s,s)$ is not identically zero. If \eqref{eq:KYPFDI} holds, then \eqref{eq:LMI} admits a positive definite solution.
\end{prop}
\begin{proof}
We show that $P = P_+$ as in \eqref{eq:maximumsolution} is positive definite. Consider a factorization of the form \eqref{eq:lure}. Note that $L$ is square from the comments following Lemma \ref{lemm:lure} and from the fact that $\Pi(-s,s)$ has full normal rank. Also, with $G(s) = L + K(sI - A)^{-1}B$,
\[
\nrank G(s) = \nrank \Pi(-s,s) = n
\]
 from the equality \eqref{eq:spectralfactorization} considered on the imaginary axis. 
 
 Suppose that $x_0$ is such that $Px^0 = 0$. Let $(x,u)$ be a trajectory as in the definition of minimal stability, with $x(0) = x^0$. Since $P \geq 0$ by minimal stability, the dissipation inequality 
 \[
\langle Px(t), x(t) \rangle - \langle Px^0, x^0 \rangle + \int_0^t |Kx(s) + Lu(s)|^2\, ds = E(x,u,0,t)
 \]
implies that $Px = 0$ and $Kx + Lu = 0$ a.e. In terms of the LTI system
  \begin{align*}
  \dot x &= Ax + Bu, \\
  v &= Kx + Lu,
  \end{align*}
  we thus have $(x,u) \in \mathcal{Z}$. In particular, $x$ is contained in $\mathcal{S}$ and $\dot x = (A+BF)x$ according to Lemma \ref{lemm:zerodynamics}. The system matrix
 \[
\begin{bmatrix}
A-sI & B \\
K & L
\end{bmatrix}
\]
is square, and Lemma \ref{lemm:lure} implies that it is invertible in the open left-half plane. Consequently $(A+BF)|_{\mathcal{S}}$ has no eigenvalues in the open left half-plane according to Lemma \ref{lemm:zerodynamicsspectrum}. Since by definition $x(t) \rightarrow 0$ as $t \rightarrow \infty$, it must be that $x(0) = x^0 = 0$.
\end{proof}

Popov showed that for hyperstability, the rank condition on $\Pi(-s,s)$ can be relaxed somewhat. Adopting the notation and hypotheses of Lemma \ref{lemm:zerodynamics}, we need an additional observation:
\[
\range B|_{\ker D} \cap \mathcal{S} = \{0\}.
\]
 Indeed, suppose to the contrary that there exists $u^0 \in \ker D$ for which $Bu^0 \in \mathcal{S}$. Set $u(t) = e^{st}u^0$. If we choose $s \in \CC$ not an eigenvalue of $(A+BF)|_{\mathcal{S}}$, then $x(t) = e^{st}(sI - A - BF)^{-1}Bu_0$ satisfies
\[
(x, u + Fx) \in \mathcal{Z}.
\]
Thus $u = 0$ identically, and in particular $u^0 = 0$. 

Since $\range B|_{\ker D} \cap \mathcal{S} = \{0\}$, we can choose an observability decomposition in the form
\begin{equation} \label{eq:zerodynamicskalman}
A+BF = \begin{bmatrix}
A_{11} & 0 \\
A_{21} & A_{22}
\end{bmatrix}, \quad B = \begin{bmatrix}
B_{11} & B_{12} \\
0 & B_{22},
\end{bmatrix}, \quad C + DF = \begin{bmatrix}
C_1 & 0
\end{bmatrix},
\end{equation}
where $(A_{11}, C_1)$ is observable. Here $A_{22} = (A+BF)|_{\mathcal{S}}$, and the block structure of $B$ is with respect to the decomposition $\RR^n = \ker D \oplus (\ker D)^\perp$. Furthermore, $B_{11}$ has full column rank: if there is $u_1 \in \ker D$ for which $B_{11}u_1 = 0$, then $G(s)$ cannot have full normal rank.

We also use that hyperstability is invariant under transformations of the form \eqref{eq:feedbackcontrol}, which follows immediately from the fact that $c^{-1}|x| \leq |x_\sim| \leq c|x|$ for some $c > 0$ and $E(x,u,0,t) = E_\sim(x_\sim,u_\sim,0,t)$.

\begin{prop}[{\cite[\S 16, Theorem 1]{popov1973hyperstability}}]  \label{prop:popovhyperstability2}
	Suppose that $B$ is nonzero, $(A,B,M)$ is minimally stable, and $\det \Pi(\eta, \zeta)$ is not identically zero. Then $(A,B,M)$ is hyperstable if and only if \eqref{eq:KYPFDI} holds.
\end{prop}
\begin{proof}
	We provide a proof when $(A,B)$ is controllable. Otherwise, one can pass to a controllable subsystem and separately estimate the residual part. Since this is relatively straightfoward and mostly unrelated to the rest of the proof, we refer the reader to \cite[\S 16]{popov1973hyperstability} for the details.
	
	According to Lemma \ref{lemm:hyperstabilityimpliesFDI}, hyperstability implies that \eqref{eq:KYPFDI} holds, so we focus on the converse. Since $(A,B)$ is controllable, \eqref{eq:KYPFDI} implies that \eqref{eq:LMI} has a Hermitian solution $P$, which is positive semidefinite by minimal stability. Consider any factorization \eqref{eq:lure}. In addition to $G(s) = L + K(sI-A)^{-1}B$, also define $H(s) = P^{1/2}(sI - A)^{-1}B$. Thus \eqref{eq:popovfactorization} can be written in the form
	\[
	\Pi(\eta,\zeta) = (\eta + \zeta)H(\bar \eta)^* H(\zeta) + G(\bar \eta)^*G(\zeta).
	\]
	By hypothesis $\det \Pi(\eta, \zeta)$ is not identically zero, which implies that
	\[
	\Psi(s) =\begin{bmatrix}
	G(s) \\
	H(s)
	\end{bmatrix}
	\]
	has full normal column rank: otherwise, for each $\zeta_0 \in \CC$ not an eigenvalue of $A$, there would exist $u^0 \in\KK^n$ such that $G(\zeta_0)u^0 = H(\zeta_0)u^0 = 0$. This implies that $\det \Pi(\eta,\zeta_0) = 0$ for all $\eta \in \CC$ not an eigenvalue of $A$, and hence the determinant is zero as a rational function of two variables.
	
 Apply Lemma \ref{lemm:zerodynamics} with the system matrices $(A,B,C,D)$, where
	\begin{equation} \label{eq:popovCD}
	C = \begin{bmatrix}
	K \\
	P^{1/2}
	\end{bmatrix}, \quad D = \begin{bmatrix}
	L \\ 0
	\end{bmatrix}.
	\end{equation}
	Note that $\Psi(s) = D + C(sI-A)^{-1}B$. From this we deduce two important facts.
	
	\begin{inparaenum}
		\item Firstly, $\ker P = \mathcal{S}$. To see this, we argue as in Proposition \ref{prop:popovhyperstability1}: if $x^0 \in \KK^m$ is such that $Px^0 = 0$, then by minimal stability there exists $(x,u) \in \mathcal{Z}$ for which $x(0) = x^0$, and hence $x^0 \in \mathcal{S}$. Conversely, if $x^0 \in \mathcal{S}$, then $(C+DF)x^0 = 0$, which implies $P^{1/2}x^0 =0$ by \eqref{eq:popovCD}.
		
		\item We also show that $(A+BF)|_{\mathcal{S}}$ is Hurwitz. Let $x^0 \in \mathcal{S}$ be an eigenvector of $(A+BF)|_{\mathcal{S}}$ with eigenvalue $\lambda \in \CC$. By minimal stability, there exists $(x,u) \in \mathcal{Z}$ such that $x(0) = x^0$ and $x(t) \rightarrow 0$ as $t\rightarrow \infty$. Thus $x(t) = e^{(A+BF)t}x^0 = e^{\lambda t}x^0$, so $\Re \lambda < 0$.
		
	\end{inparaenum}
	
	Since hyperstability is invariant under transformations of the form \eqref{eq:feedbackcontrol}, we now proceed to consider the system obtained by replacing $u$ with $u_\sim = u - Fx$ (we also freely consider transformations $x \mapsto T^{-1}x$, but this is not reflected in the notation). Fix a decomposition \eqref{eq:zerodynamicskalman}, and partition 
	\[
	x = (x_1, x_2), \quad u_\sim = (u_1, u_2)
	\]
	accordingly, where $x_2 \in \mathcal{S}$ and $u_1 \in \ker D$. Thus
	\begin{align*}
	\dot  x_1 &= A_{11}x_1 + B_{11}u_1 + B_{12}u_2, \\
	\dot x_2 &= A_{21}x_1 + A_{22}x_2 + B_{22}u_2
	\end{align*}
Suppose that $E_\sim(x,u_\sim,0,t) \leq \beta^2$ for all $t \geq 0$. Since $\ker P = \mathcal{S}$, the dissipation inequality implies that
\begin{equation} \label{eq:observableiscontrolled}
|x_1(t)| \leq c(|x(0)| + \beta)
\end{equation}
for some $c > 0$. It thus remains to estimate $|x_2(t)|$, where
\[
x_2(t) = e^{A_{22}t}x_2(0) + \int_0^t e^{A_{22} (t-s)} \left( A_{21}x_1(s) + B_{22} u_2(s) \right) \, ds.
\]
 In light of \eqref{eq:observableiscontrolled} and the fact that $A_{22} = (A+BF)|_{\mathcal{S}}$ is Hurwitz, the only troublesome term is
\begin{equation} \label{eq:variationofparameters}
\int_0^t e^{A_{22} (t-s)}  B_{22} u_2(s) \, ds.
\end{equation}
According to Lemma \ref{lemm:zerodynamics}, $\range(C+DF)$ is orthogonal to $\range D$, so $\range (K + LF)$ and $\range L$ are also orthogonal by \eqref{eq:popovCD}. Now the dissipation inequality implies that
\[
\int_{0}^t |(K+LF)x(s) + Lu_\sim(s)|^2 \, ds \leq \beta^2 + |P^{1/2}x(0)|^2.
\]
By orthogonality, $|(K +LF)x + Lu_\sim|^2 \geq |Lu_\sim|^2 \geq |u_2|^2/c$ for some $c>0$. Thus
\[
\int_0^t |u_2(s)|^2 \, ds \leq c (|x(0)|^2 + \beta^2)
\]
for some $c>0$. By Cauchy--Schwarz, \eqref{eq:variationofparameters} can be bounded (in absolute value) by a multiple of $|x(0)| + \beta$, as desired.
\end{proof}

There are cases where Proposition \ref{prop:popovhyperstability1} is not applicable, but where nevertheless observability conditions can be used to show that every Hermitian solution of $\Lambda(P) \leq 0$ is nonsingular. For an example where this is useful, see \S \ref{subsect:tau>0proof}.

 \begin{lemm} \label{lemm:barabanov}
 	Suppose that $P = P^*$ satisfies $\Lambda(P) \leq 0$. If $(A,S)$ is observable and there exists $K \in \KK^{n\times n}$ such that 
 	\[
 	Q + 2\Re S^* KS + (KS)^*R(KS) \leq 0, \quad \det (I + RK) \neq 0,
 	\]
 	then $\det P \neq 0$. 
 \end{lemm}
 \begin{proof}
 	First assume that $Q \leq 0$. Given $\delta > 0$, apply \eqref{eq:feedbackcontrol} with $F = -\delta S$. From the upper left block of the resulting inequality $\Lambda_\sim(P) \leq 0$,
 	\[
 	(A-\delta BS)^*P + P(A-\delta BS)\leq  -S^*(2\delta I - \delta^2 R)S.
 	\]
 	Take any $\delta \in (0,2/\|R\|)$, and note that $(A-\delta B {S}, S^*(2\delta I - \delta^2 R)S)$ is observable. Since the right-hand side is negative semidefinite, standard facts about the Lyapunov equation imply that $\det P \neq 0$. In general, first apply the feedback \eqref{eq:feedbackcontrol} with $F = KS$. The previous argument applies provided the pair $(A + BKS, (I+RK)S)$
 	is observable, which is true if, e.g., $I+RK$ is invertible.
 \end{proof}
 
\subsection{Stability criteria}

 \label{subsect:expstability} In this section we record two frequency conditions for exponential stability which are analogues of the classical circle and Popov criteria. First we record a version of the circle criterion that applies to systems of the form $\dot z = Az + B\phi(Cz)$; the proof is standard, and is essentially unaffected by the inclusion of the exponential rate $r$. Let 
\begin{equation} \label{eq:transferfunction}
H(s) = C(s-A)^{-1}B
\end{equation}
denote the transfer matrix of the associated linear block. Also define
\begin{equation} \label{eq:circlepopovfunction}
\Phi(\eta,\zeta) = H(\eta)^* + H(\zeta) - 2L^{-1}I
\end{equation}
for use in the next lemma.

\begin{lemm} \label{lemm:circle}
	Let $r \geq 0$, and let $\phi : \KK^n\rightarrow \KK^n$ belong to the sector $[0,L]$, where $0 < L \leq \infty$. Suppose that the following conditions hold.
	\begin{enumerate} \itemsep6pt 
		\item $B \neq 0$, 
		\item  $A + k BC + rI$ is Hurwitz for some $k \in [0,L]$, 
		\item $\det \Phi(\eta,\zeta)$ is not identically zero, where $\Phi(\eta,\zeta)$ is given by \eqref{eq:circlepopovfunction}.
	\end{enumerate}
	If the frequency inequality
		\begin{equation} \label{eq:circleFDI}
	\Re H(i\omega -r) \leq L^{-1}I
	\end{equation}
	holds whenever $\omega \in \RR$ and $\det(A - (i\omega - r)I) \neq 0$, then the system $\dot z = Az + B\phi(Cz)$ is $r$-exponentially stable. 
\end{lemm}
\begin{proof}
Given $r \geq 0$, consider the LTI system
\begin{equation}  \label{eq:LTIweighted}
\begin{aligned}
\dot x &= (A+rI)x + Bu, \\
y &= Cx.
\end{aligned}
\end{equation} 
Define the supply rate $\sigma(x,u) = - \Re \langle u,y\rangle +L^{-1}|u|^2$,
in which case
\[
M = \begin{bmatrix}
Q & S^* \\
S & R
\end{bmatrix}  = 
\tfrac{1}{2}\begin{bmatrix}
0 & -C^* \\
-C & 2 L^{-1}I
\end{bmatrix}.
\]	
According to Lemma \ref{lemm:minimal}, the Hurwitz assumption on $A + kBC + rI$ implies that $(A+rI, B, M)$ is minimally stable: since $k \in [0,L]$, the control $u = kCx$ satisfies $E(x,u,0,t) \leq 0$. The corresponding Popov function is 
\begin{equation}
\Pi(\eta, \zeta) = -\tfrac{1}{2} \Phi(\bar \eta - r, \zeta -r).
\end{equation}
If $\det \Phi(\eta,\zeta)$ does not vanish identically, then the hypotheses of Proposition \ref{prop:popovhyperstability2} are satisfied, and hence $(A+rI,B,M)$ is hyperstable. 

Now suppose that $\dot z = Az + B\phi(Cz)$. If we set $x = e^{rt}z$ and $u = e^{rt}\phi(e^{-rt}y)$, then the pair $(x,u)$ satisfies $\dot x = (A+rI)x + Bu$. Furthermore, by the sector condition, $(x,u)$ satisfies
\[
E(x,u,0,t) \leq 0 \text{ for all $t \geq 0$}.
\]
Since $(A+rI,B,M)$ is hyperstable, there exists $c > 0$ such that 
\[
e^{rt}|z(t)| \leq  c |z(0)|
\] 
whenever $\dot z = Az + B\phi(Cz)$.
\end{proof}

 Next, suppose that $\KK = \RR$. We give a Popov-type criterion for the exponential stability of systems $\dot z = Az + B\nabla f(Cz)$. This extends a result of Bliman \cite{bliman2002absolute} albeit under slightly different hypotheses. When $f$ is quasi-strongly convex, the criterion presented below is potentially less conservative than the one in \cite{bliman2002absolute}, and applies even if $\nabla f$ belongs to an infinite sector (which is useful for the applications in \S \ref{sect:dissipativehamiltonian}). Set
\begin{equation}
\begin{multlined} \label{eq:popovpopovfunction}
\Phi(\eta,\zeta) = (\lambda +  \mu\bar \eta+ \nu(\bar \eta + 2r)  )H(\eta)^* + (\lambda + \mu \zeta+  \nu(\zeta+2r))H(\zeta) \\ + 2\mu rLH(\eta)^*H(\zeta) - 2 \lambda L^{-1} I ,
\end{multlined}
\end{equation}
where $\lambda,\, \mu,\, \nu \in \RR$ are parameters and $H(s)$ is given by \eqref{eq:transferfunction}. Bliman considered the case $\nu = 0$, but where $f$ is not necessarily quasi-strongly convex.

\begin{lemm} \label{lemm:popov}
	Let $r >  0$, and let $f : \RR^n \rightarrow \RR$ be such that $\nabla f$ belongs to the sector $[0,L]$, where $0 < L \leq \infty$. Let $\lambda, \, \mu, \, \nu \geq 0$, and suppose that the following conditions are satisfied.
\begin{enumerate} \itemsep6pt 
	\item $B \neq 0$,
	\item $A + k BC + r I$ is Hurwitz for some $k \in [0,L]$,
	\item $\det \Phi(\eta,\zeta)$ is not identically zero, where $\Phi(\eta,\zeta)$ is given by \eqref{eq:popovpopovfunction}, 
	\item $\mu = 0$ if $L = \infty$, 
    \item $\nu =0$ if $f$ is not $0$-strongly quasi-convex.
\end{enumerate}
If the frequency inequality
	\begin{equation} \label{eq:popovFDI}
	\begin{multlined} 
	\Re\left[ (\lambda + \mu(i\omega-r) +\nu(i\omega+r)H(i\omega-r) \right]  + \mu Lr H(i\omega-r)^*H(i\omega-r)  \\ \leq \lambda L^{-1} I
	\end{multlined} 
	\end{equation}
	holds whenever $\omega \in \RR$ and $\det(A - (i\omega-r)I) \neq 0$, then $\dot z = Az + B\nabla f(Cz)$ is $r$-exponentially stable.
\end{lemm}
\begin{proof}
	As in the proof of Lemma \ref{lemm:circle}, we consider the LTI system $\dot x = (A+ rI)x + Bu$ and set $y = Cx$. By a slight abuse of notation, we also introduce the additional output
	\[
	\dot y = C(A+rI)x + CBu.
	\]
	If $\dot x = (A+rI)x + Bu$, then $\dot y$ defined in this way is indeed the time derivative of the function $y = Cx$. Define the following supply rates:
	\begin{align*}
	\sigma_0(x,u) &= - \langle u, y \rangle + L^{-1}|u|^2  \\
    \sigma_1(x,u) &= - \langle u, \dot y - ry \rangle -rL|y|^2 , \\
    \sigma_2(x,u) &= - \langle u, \dot y + ry\rangle.
	\end{align*}
 The matrices defining these quadratic forms are
  	\begin{align*}
  	M_0 &= -\tfrac{1}{2}\begin{bmatrix}
  	0 & C^* \\
  	C & -2L^{-1} I 
  	\end{bmatrix},\\ 
  	M_1 &= - \tfrac{1}{2}\begin{bmatrix}
  	rL C^*C & A^*C^* \\
  	CA & 2 \Re CB
  	\end{bmatrix}, \\  	
  	M_2 &= -\tfrac{1}{2}\begin{bmatrix}
  	0 & (A+2rI)^*C^* \\
  	C(A+2rI) & 2 \Re CB
  	\end{bmatrix}.
  	\end{align*}
  	  	Finally, set $\sigma = \lambda \sigma_0 + \mu \sigma_1 + \nu \sigma_2$ and $M = \lambda M_0 + \mu M_1 + \nu M_2$. We show that $(A+rI, B, M)$ is hyperstable using Proposition \ref{prop:popovhyperstability2}. The Popov function for this system is
  	  	\[
  	  	\Pi(\eta,\zeta) = -\tfrac{1}{2} \Phi(\bar \eta -r, \zeta-r),
  	  	\]
  	  	whose determinant is not identically zero by hypothesis. Furthermore, \eqref{eq:popovFDI} implies that \eqref{eq:KYPFDI} holds. It remains to show that $(A+rI,B,M)$ is minimally stable. For this we adapt another argument due to Popov \cite[\S 25]{popov1973hyperstability}.
  	  	
  	  	Given $x^0 \in \RR^m$, we seek a pair $(x,u)$ satisfying the conditions of minimal stability by first solving the system
  	  	\[
  	  	\dot x = (A+rI + kBC)x + B\rho, \quad x(0) = x^0
  	  	\]
  	  	with an (as of yet) unspecified function $\rho$, and then setting $u = kCx + \rho$. If $\rho \in L^2(\RR_+; \RR^n)$, then $x(t) \rightarrow 0$ as $t\rightarrow \infty$ since $A+rI + kBC$ is Hurwitz. 

        We can assume that $k > 0$, otherwise it suffices to choose $u = 0$ as the control. If $\rho$ has a locally square-integrable derivative, then $\sigma(x,u)$ can be written entirely in terms of $(u, \rho)$ by replacing $y$ with $(u- \rho)/k$ and $\dot y$ with $(\dot u -\dot \rho)/k$. Thus
  	    \[
  	    \sigma_0(x,u) = k^{-1}\langle u, \rho \rangle + (L^{-1} - k^{-1})|u|^2 \leq k^{-1} \langle u, \rho \rangle. 
  	    \]
  	    Similarly,
  	    \begin{align*}
  	    \sigma_1(x,u) &= k^{-1}\langle u, \dot \rho  - \dot u \rangle + rk^{-1}\langle u, u-\rho\rangle - rLk^{-2}|u-\rho|^2 \\
  	    &\leq k^{-1} \langle u, \dot \rho -\dot u\rangle + rk^{-1}\langle u, \rho \rangle,
  	    \end{align*}
       and 
        \begin{align*}
        \sigma_2(x,u) &= k^{-1}\langle u, \dot \rho - \dot u\rangle + rk^{-1}\langle u, \rho - u \rangle  \\ 
        &\leq k^{-1}\langle u,\dot \rho -\dot u\rangle + rk^{-1}\langle u,\rho \rangle.
        \end{align*}
     Altogether then,
     \begin{equation} \label{eq:popovsupplybound}
     \sigma(x,u) \leq  k^{-1}(\mu+\nu) \langle u, \dot \rho \rangle +  k^{-1}(\lambda + r(\mu + \nu))\langle u, \rho \rangle - k^{-1}\langle u,\dot u\rangle.
     \end{equation}
     We can assume that $\mu + \nu > 0$, otherwise we are in the situation of Lemma \ref{lemm:circle} where it suffices to let $\rho = 0$. The first two terms on the right-hand side of \eqref{eq:popovsupplybound} vanish if we choose $\rho$ to be any solution of
     \[
     (\mu + \nu)\dot \rho + (\lambda + r(\mu+\nu))\rho = 0.
     \]
     Since $r > 0$, any such solution is square-integrable. With this choice,
     \[
     E(x,u,0,t) \leq -k^{-1} \int_0^t \langle u(s), \dot u(s) \rangle \, ds \leq \tfrac{1}{2} k^{-1} |u(0)|^2.  
     \]
     By specifying the initial condition $\rho(0) = -kCx^0$, we arrange that $u(0)= 0$, and hence $E(x,u,0,t) \leq 0$ as desired.
  
  Now suppose that $\dot z = Az + B\phi(Cz)$. As in the proof of Lemma \ref{lemm:circle}, we set $x = e^{rt}z$ and $u = e^{rt}\phi(e^{-rt}y)$. The sector condition implies that 
  	\[
  	E_0(x,u,0,t) \leq 0 \text{ for all $t \geq 0$},
  	\]
  	 where the left-hand side is defined in the obvious way. Next, note that by replacing the potential $f$ with $f - f(0)$, we can assume that $f(0) = 0$ and $f \geq 0$. Now compute
  	 \begin{equation} \label{eq:FTC}
  	  e^{2rt}f(e^{-rt} y(t)) - f(y(0)) = \int_0^t 2re^{2rs}f(e^{-rs}y(s)) + \langle  u(s), \dot y(s) - ry(s) \rangle \, ds.
  	 \end{equation}
 If $L < \infty$, then the sector condition and $f(0) = 0$ imply that $f(w) \leq (L/2)|w|^2$, which can be used to bound the first term under the integral in \eqref{eq:FTC}.  If $f$ is $0$-strongly quasi-convex, then we instead use $f(w) \leq \langle \nabla f(w), w \rangle$. In either case
  	 \[
  	 E_j(x,u,0,t) \leq f(y(0)), \quad j = 1,2.
  	 \]
  	 Combining these results,
  	  \[
  	  E(x,u,0,t) \leq (\mu + \nu)f(y(0)) \text{ for all $t \geq 0$}.
  	  \]
	 Since $(A + rI, B,M)$ is hyperstable, there exists $c>0$ such that
	 \[
	 e^{2rt} |z(t)|^2 \leq c(|z(0)|^2 + f(Cz(0)))
	 \] 
	 whenever $\dot z = Az + B\nabla f(Cz)$, which implies $r$-exponential stability.
\end{proof} 	

In the context of Lemma \ref{lemm:circle}, if $L< \infty$, then the $\det \Pi(\eta,\zeta)$ is not identically zero provided both $|\eta|$ and $|\zeta|$ are sufficiently large.  If the first Markov parameter $CB$ vanishes, then the same is true with regards to Lemma \ref{lemm:popov}, since it implies that any linear multiple of $H(s)$ is strictly proper.

If $(A,B)$ is controllable and the hypotheses of Proposition \ref{prop:popovhyperstability1} are satisfied in Lemma \ref{lemm:circle}, then there exists a quadratic Lyapunov function $V(z) = \langle Pz,z\rangle$ with $P > 0$ such that 
\[
\dot V\leq 2r V
\] 
along trajectories of $\dot z = Az + B\phi(Cz)$. Similarly, for Lemma \ref{lemm:popov} we obtain the existence a Lyapunov function $V(z) = \langle Pz, z\rangle + f(Cz) - f(0)$.

\section{Dissipative Hamiltonian systems} \label{sect:dissipativehamiltonian}

\subsection{Preliminaries} \label{subsect:dissipativeprelim}
 Let $x = (q,p)$ be coordinates on the phase space $\RR^{2d} = \RR^d \times \RR^d$, where $d\geq 1$. Given a $\mathcal{C}^1$ function $f : \RR^d \rightarrow \RR$, define the Hamiltonian
\[
H(x) = \tfrac{1}{2}|p|^2 + f(q) - f(0).
\]
By replacing $f$ with $f-f(0)$ we will assume that $f(0) = 0$. Assume $\nabla f$ belongs to a sector $[m,L]$ for some $0 < m < L \leq \infty$ (see \S \ref{subsect:applicationscontinuous}). Consider the system
\begin{equation*}
\dot x = (J-S)\nabla H(x),
\end{equation*}
where $J$ is the standard symplectic matrix,
and $S \in \RR^{2d\times 2d}$ is positive semidefinite. Partition
\[
S = \begin{bmatrix}
S_{11} & S_{12} \\
S_{21} & S_{22}
\end{bmatrix},
\]
where $S_{11},S_{22}$ are positive semidefinite and $S_{12}^* = S_{21}$. Rewrite this system in the Lur'e form
\begin{equation} \label{eq:loopshiftedequation}
\begin{aligned}
\dot x &= Ax + Bu, \\
u &= \nabla g (Cx),
\end{aligned}
\end{equation} 
where $g(q) = f(q) - (m/2)|q|^2$, and the matrices $A,B,C$ are given by
\[
A  = \begin{bmatrix}
-m S_{11} & I - S_{12} \\
-m(I  + S_{21}) & - S_{22}
\end{bmatrix}, \quad B = \begin{bmatrix}
-S_{11} \\
-I - S_{21}
\end{bmatrix}, \quad C = \begin{bmatrix}
I & 0 
\end{bmatrix}.
\]
Next, assume that $S_{12} = 0 = S_{21}$. In that case it is easy to see  $(A,B)$ is controllable and $(A,C)$ is observable. Finally, assume that $S_{11},S_{22}$ are multiples of the identity
\[
S_{11} = \tau I , \quad S_{22} = 2\sigma I,
\]
where $\sigma > 0$ and $\tau \geq 0$.

 From the resulting block structure, we can assume that $d=1$  as far as verifying the hypotheses of Lemmas \ref{lemm:circle} or \ref{lemm:popov} is concerned. Thus
\begin{equation} \label{eq:transferrunctionHO}
H(s) = \frac{1+2\sigma \tau +\tau s}{\det(sI-A)} = \frac{-(1+2\sigma \tau +\tau s)}
{s^2 + (2\sigma+ \tau m) s +m(1+2\sigma \tau)},
\end{equation}
where $H(s) = C(sI - A)^{-1}B$ is the scalar transfer function. Similarly, the Popov function in either Lemma \ref{lemm:circle} or \ref{lemm:popov} can be assumed scalar.
\subsection{The linear problem} 

 \label{subsect:marginalstabilitysector}
In order for the system \eqref{eq:H} to be $r$-exponentially stable when $\nabla f$ belongs to a sector $[m,L]$, it is necessary for
\begin{equation} \label{eq:marginallystable}
 A + (k-m)BC + rI = \begin{bmatrix}
r-k\tau  & 1 \\
-k & r-2\sigma 
\end{bmatrix}
\end{equation}
to be marginally stable whenever $k \in [m,L]$. By a marginally stable matrix we mean one whose eigenvalues lie in the closed left half-plane, and whose purely imaginary eigenvalues are semisimple. The characteristic polynomial of \eqref{eq:marginallystable} is
\begin{equation} \label{eq:charpoly}
\chi(s) =  s^2 + (2\sigma -2r + k\tau)s+ k+r^2 -2r\sigma + (2\sigma - r)k\tau.
\end{equation}
For the roots of $\chi$ to lie in the closed left half-plane, the linear and constant terms of $\chi$  must be nonnegative. 
First we record the observations needed to analyze \eqref{eq:H0}.

\begin{lemm} \label{lemm:tau=0linearconstraint} If $\tau = 0$, then \eqref{eq:marginallystable} is marginally stable for every $k\geq m$ if and only if 
	\[
	r \leq \sigma, \quad m \geq 2r\sigma -r^2,
	\] and at least one of these inequalities is strict. Furthermore, the following hold.
	\begin{enumerate} \itemsep6pt 
		\item 	 If $r < \sigma$ and $m > 2r\sigma -r^2$, then $A+rI$ is Hurwitz.
		\item If $r < \sigma$ and $m  = 2r \sigma -r^2$, then $A+rI$ has distinct eigenvalues $0$ and $2(r-\sigma)$, and $A+\delta BC + rI$ is Hurwitz for any $\delta > 0$.
	\end{enumerate}
\end{lemm}
\begin{proof}
	This follows directly from \eqref{eq:charpoly}.
\end{proof}

If $r = \sigma$ and $m > 2r\sigma -r^2$, then $A+rI$ has distinct eigenvalues 
$\pm i (m-\sigma^2)^{1/2}$,
but $A+\delta BC + rI$ is not Hurwitz for any $\delta \geq 0$. 

\begin{lemm} \label{lemm:tau>0constraint}
If $\tau > 0$ and \eqref{eq:marginallystable} is marginally stable for each $k \geq m$, then
\begin{equation} \label{eq:tau>0constraint}
r \leq \min(2\sigma + \tau^{-1}, \sigma + m \tau /2), \quad m\tau( 2 \sigma +\tau^{-1} - r)\geq 2r\sigma - r^2.
\end{equation}
If $r = 2\sigma + \tau^{-1} = \sigma + \tau m /2$, then $A+ r I$ has distinct imaginary eigenvalues, and $A+ \delta BC + rI$ is Hurwitz for any $\delta > 0$.
\end{lemm}
\begin{proof}
	The constraints \eqref{eq:tau>0constraint} come from the roots of \eqref{eq:charpoly} for $k=m$ and in the limit as $k\rightarrow \infty$. If $r = \sigma + \tau m/2$, then the linear term in $\chi$ vanishes for $k = m$; if additionally $r = 2\sigma + \tau^{-1}$, then the constant term is $r^2 - 2r\sigma$ for $k = m$, which is strictly positive since 
	\[
	r^2 = (2\sigma + \tau^{-1})r > 2\sigma r.
	\] The Hurwitz statement follows because both the linear and constant terms in \eqref{eq:charpoly} are strictly increasing with $k$ when $\tau > 0$.
\end{proof}

Thus the optimal pair $(r_\star,\tau_\star)$ for the linear problem over the sector $[m,\infty]$ is obtained by solving $r = 2\sigma +\tau^{-1} = \sigma + \tau m /2$. The resulting values are
\begin{equation} \label{eq:optimalinfinite}
r_\star=\frac{3\sigma + \sqrt{2 m+\sigma^2} }{2}, \quad \tau_\star = \frac{\sigma + \sqrt{2 m+\sigma^2}}{m},
\end{equation}
as in the statement of Theorem \ref{theo:tau>0}.

\subsection{Proof of Theorem \ref{theo:H0finitesector}} \label{subsect:H0finitesectorproof}
Theorem \ref{theo:H0finitesector} is an application of Lemma \ref{lemm:circle}. Suppose that
\[
r < \sigma, \quad m \geq 2r\sigma - r^2. 
\]
Given $m < L \leq \infty$, set $l = L-m$. The relevant frequency inequality \eqref{eq:circleFDI} is
\begin{equation} \label{eq:H0finitesectorFDI}
\Re H(i\omega - r) \leq 1/l. 
\end{equation}
Let $\mathcal{F}$ denote the set of $l > 0$ such that \eqref{eq:H0finitesectorFDI} holds whenever $\omega \in \RR$ and $i\omega -r$ is not an eigenvalue of $A$. Our aim is to compute
\[
l_{\sup} = \sup \mathcal{F}.
\]
Later we will verify the remaining hypotheses of Lemma \ref{lemm:circle}.

After placing over the common denominator $l|\det(A + rI -i\omega I)|^2$, the inequality \eqref{eq:H0finitesectorFDI} is equivalent to
\begin{equation} \label{eq:H0finitesectorquadratic}
\omega^4 + \beta\omega^2 + \gamma \geq 0 \text{ for all }\omega \in \RR,
\end{equation}
where the coefficients are given by
\begin{align*}
\gamma &= (m - 2r\sigma +r^2)(l + m - 2r\sigma +r^2), \\
\beta &= 2(\sigma^2-m) + 2(\sigma - r)^2 - l.
\end{align*}
Given that the polynomial in \eqref{eq:H0finitesectorquadratic} is quadratic in $\omega^2$, define the following two feasibility sets:
\begin{gather}
\mathcal{F}_1 = \{l > 0: \gamma > 0 \text{ and } 4\gamma -\beta^2 \geq 0\}, \label{eq:Q1}\\
\mathcal{F}_2 = \{l>0: \gamma \geq 0 \text{ and } \beta \geq 0\}. \label{eq:Q2}
\end{gather}
These sets depend implicitly on $m,\, \sigma, \, r$, and $\mathcal{F} = \mathcal{F}_1 \cup \mathcal{F}_2$.
\begin{lemm} \label{lemm:H0finitesectoraux}
If $r < \sigma$ and $m \geq 2r\sigma -r^2$, then
\[
l_{\sup} =  4\left((\sigma-r)^2 + (\sigma-r)(m-2r\sigma + r^2)^{1/2} \right).
\]
\end{lemm}
\begin{proof}
	\begin{inparaenum}
 \item If $m = 2r\sigma -r^2$, then $\gamma = 0$. Thus \eqref{eq:H0finitesectorquadratic} holds if and only if $l \in \mathcal{F}_2$, namely $\beta \geq 0$. But $\beta \geq 0$ if and only if $l \leq 4(\sigma-r)^2$, so
 \[
 l_{\sup} = 4(\sigma-r)^2.
 \]
 
 \item If $m > 2r\sigma -r^2$, then $\gamma > 0$ for all $l \geq 0$. Define $h = 4\gamma -\beta^2$, which is a concave quadratic function of $l$:
 \[
 h(l) = -l^2+8 (\sigma -r)^2 l +16 (m-\sigma ^2) (\sigma -r )^2.
 \]
The discriminant of $h$ with respect to $l$ is
\[
\disc(h; l) = 64 (\sigma -r )^2 (m- 2r\sigma + r^2) > 0,
\]
so $h$ has two real roots $l_\pm$. Furthermore, $h'(0) = 8(\sigma -r)^2  > 0$, so $l_+ > 0$. Explicitly,
 \[
 l_+ = 4\left((\sigma-r)^2 + (\sigma-r)(m-2r\sigma + r^2)^{1/2} \right).
 \]
 Now $\beta$ is a decreasing function of $l$, and clearly $h > 0$ at a positive root of $\beta$ (if one exists) since $\gamma > 0$. Thus $l_+$ exceeds any positive root of $\beta$, which implies that
 \[
l_+ = \sup \mathcal{F}_1 \geq \sup \mathcal{F}_2,
 \]
 and hence $l_+ = \max( \sup \mathcal{F}_1, \sup \mathcal{F}_2) = l_{\max}$.
	\end{inparaenum} 
\end{proof}

To finish the proof of Theorem \ref{theo:H0finitesector}, we must verify the three hypotheses of Lemma \ref{lemm:circle}: 
\begin{inparaenum} \item Clearly $B \neq 0$, and in fact, $(A,B)$ is controllable.
	 \item According to Lemma \ref{lemm:tau=0linearconstraint}, $A+rI + \delta BC$ is Hurwitz for arbitrarily small $\delta > 0$. \item Since we are assuming $l \leq l_{\max}$ and $l_{\max}$ is finite, the comments following Lemma \ref{lemm:popov} show that $\Phi(\eta,\zeta)$ is not identically zero. Furthermore, since $(A,B)$ is controllable, there exists a quadratic Lyapunov function establishing $r$-exponential stability in the sector $[m, m+l_{\sup}]$. \qed
\end{inparaenum}

\subsection{Proof of Theorem \ref{theo:H0finitesectortimeinvariant}} \label{subsect:H0finitesectortimeinvariantproof}
 Recall here the parametric assumptions $r < \sigma$ and $m \geq 2r\sigma -r^2$. The only assumption on $f$ is that $\nabla f$ belongs to a finite sector $[m,L]$. As will be apparent from the proof, there is no loss in assuming that
\[
m > 2r\sigma -r^2,
\]
since if equality holds, then one can do no better than Theorem \ref{theo:H0finitesector}. The proof of Theorem \ref{theo:H0finitesectortimeinvariant} is an application of Lemma \ref{lemm:popov} with $\nu = 0$. If $l = L-m$, then the frequency inequality \eqref{eq:popovFDI} is
\begin{equation} \label{eq:H0finitesectortimeinvariantFDI}
\Re \left[ (\lambda + \mu(i\omega -r))H(i\omega -r)\right] + \mu r l  |H(i\omega -r)|^2 \leq \lambda/l
\end{equation}
For a given $\mu \geq 0$, let $\mathcal{F}(\mu)$ denote the set of $l > 0$ such that \eqref{eq:H0finitesectortimeinvariantFDI} holds whenever $\omega \in \RR$ and $i\omega -r$ is not an eigenvalue of $A$. Equivalently, $l \in \mathcal{F}(\mu)$ if and only if
\begin{equation} \label{eq:H0finitesectortimeinvariantquadratic}
\lambda \omega^4 + \beta \omega^2 + \gamma \geq 0 \text{ for all } \omega \in \RR,
\end{equation}
where the coefficients are given by
\begin{align*}
\gamma &= -(l + m  - 2r\sigma + r^2)(l r \mu - \lambda (m- 2r\sigma + r^2)),   \\
\beta &=  2\lambda (\sigma^2-m) + 2 \lambda (\sigma -r)^2 -l(\lambda - \mu(2\sigma-r) ). 
\end{align*} 
If $\lambda = 0$, then \eqref{eq:H0finitesectortimeinvariantquadratic} fails to hold near $\omega = 0$ unless $\mu = 0$ as well. Since $\lambda = \mu = 0$ is not a relevant case, we therefore rescale $\lambda = 1$. Henceforth, consider $\gamma, \, \beta$ as functions of both $l$ and $\mu$.

For a fixed $\mu \geq 0$, define $\mathcal{F}_1(\mu)$ and $\mathcal{F}_2(\mu)$ as in \eqref{eq:Q1} and \eqref{eq:Q2}, respectively. Also define the relaxed feasible set 
\[
\mathcal{F}'_1(\mu) = \{l > 0: \gamma \geq 0 \text{ and } 4\gamma -\beta^2 \geq 0 \}.
\]
Finally, let $\mathcal{F}, \, \mathcal{F}_1,\, \mathcal{F}'_1,\, \mathcal{F}_2$ denote the unions of $\mathcal{F}(\mu), \, \mathcal{F}_1(\mu), \,\mathcal{F}'_1(\mu),\, \mathcal{F}_2(\mu)$ over $\mu \geq 0$, respectively. The goal is to compute 
\[
l_{\sup} = \sup \mathcal{F}.
\]
 As in the proof of Lemma \ref{lemm:H0finitesectoraux}, define $h = 4\gamma - \beta^2$.
 
 \begin{lemm} \label{lemm:relaxed1}
$\sup \mathcal{F}'_1 \geq \sup \mathcal{F}_2$.
\end{lemm} 
\begin{proof} 
Since $m > 2r\sigma -r^2$, we already have $\sup \mathcal{F}_2(0) \leq \sup \mathcal{F}_1(0) = \sup \mathcal{F}'_1(0)$ by the results of \S \ref{subsect:H0finitesectorproof}. Assume now that $\mu > 0$, in which case the inequality
 $\gamma(l,\mu) \geq 0$ is equivalent to
\begin{equation} \label{eq:cconstraint} 
l \leq \gamma_0/\mu, \quad \gamma_0 = \frac{m-2r\sigma + r^2}{r} > 0.
\end{equation}
Also write $\beta(l,\mu)$ in the form
\[
\beta(l,\mu) = \beta_0 - l( 1- \beta_1\mu),
\]
where $\beta_1 > 0$ but the sign of $\beta_0$ depends on the underlying parameters. Consider the zero sets of $\gamma$ and $\beta$ in the quadrant $\{ l >0,\, \mu>0\}$.

\begin{inparaenum}
	\item If $\beta_0 <0$, then the zero sets intersect at most once, and any possible intersection $(l_0, \mu_0)$ must satisfy $\mu_0 > \beta_1^{-1}$. If there is an intersection at $(l_0,\mu_0)$, then 
	\[
	\gamma(l_0,\mu_0) = \beta(l_0, \mu_0) = 0 \Longrightarrow h(l_0,\mu_0) = 0,
	\] 
	and hence $l_0 \in \mathcal{F}_1'$. Furthermore, $\mathcal{F}_2(\mu) = \emptyset$ for $\mu < \mu_0$. Since $\gamma_0/\mu$ is a decreasing function of $\mu$, any feasible $l \in \mathcal{F}_2$ must satisfy $l \leq l_0$, which proves the result. If there is no intersection, then $\mathcal{F}_2 = \emptyset$.
	
	\item If $\beta_0 > 0$, then the zero sets always intersect at some $(l_0, \mu_0)$ with $\mu_0 < \beta_1^{-1}$. This is true since $\beta_0/(1-\beta_1\mu)$ is increasing for  $\mu \in [0,\beta_1^{-1})$ and grows unboundedly as $\mu \rightarrow \beta_1^{-1}$, while $\gamma_0/\mu$ is decreasing. Thus $l_0 \in \mathcal{F}'_1$. Since $\beta_0/(1-\beta_1\mu)$ is increasing and $\gamma_0/\mu$ is decreasing, any $l \in \mathcal{F}_2$ must satisfy $l \leq l_0$.
	
	\item If $\beta_0 = 0$, then the zero set of $\beta$ is a vertical line at $\mu_0 = \beta_1^{-1}$, which always intersects the graph of $\gamma_0/\mu$ at some $(l_0, \mu_0)$. This time $\mathcal{F}_2(\mu) = \emptyset$ for $\mu < \mu_0$. Finally, argue as in the first part.
\end{inparaenum}
\end{proof}

Next, we compute $\sup \mathcal{F}'_1$. An explicit expression for $h$ is
\begin{equation}
\begin{aligned} \label{eq:hexplicit}
h(l,\mu)  =  &-(1 + (6r-4\sigma)\mu + (2\sigma -r)^2 \mu^2)l^2 \\
& + 8 (\sigma-r)(\sigma - r + (m-2\sigma^2 + r\sigma)\mu) l  \\
& +16(\sigma -r)^2 (m-\sigma^2).
\end{aligned} 
\end{equation}
This is a concave quadratic function of $l$. To see this, note that the coefficient of $l^2$ is a concave quadratic function of $\mu$ whose discriminant with respect to $\mu$ is $32 r(r-\sigma) < 0$. Since the coefficient of $l^2$ is $-1$ when $\mu =0$, it is negative for all $\mu$. The discriminant of $h$ with respect to $l$ is also quadratic in $\mu$, and
\[
\partial_\mu^2 \disc(h;l) = 128m(\sigma-r)^2(m-2r\sigma + r^2) > 0,
\]
so this discriminant is convex. Furthermore,
\[
\disc(\disc(h; l); \mu) = 16384(m-r^2)^2(\sigma -r )^2(\sigma^2- m).
\]
When $m \leq \sigma^2$, let $\mu_- \leq \mu_+$ denote the roots of $\mu \mapsto \disc(h; l)$, and define
\begin{equation} \label{eq:I}
I_\mu = \begin{cases}
(-\infty,\mu_-] \cup [\mu_+, \infty) & \text{if } m \leq \sigma^2, \\
\RR &\text{if }  m > \sigma^2 . \\
\end{cases}
\end{equation}
Thus $h(l,\mu)$ has real roots $l_\pm(\mu)$ for each $\mu \in I_\mu$. Moreover, $l_\pm(\mu)$ is a smooth function of $\mu$ for $\mu \in J_\mu$, where
\begin{equation} \label{eq:J}
J_\mu = \begin{cases} (-\infty,\mu_-) \cup (\mu_+,\infty) &\text{if } m \leq \sigma^2 \\
\RR & \text{if } m > \sigma^2. \end{cases}
\end{equation}

Since $l_+(\mu)$ is the larger of the two roots and $l \mapsto h(l,\mu)$ is a concave quadratic, it follows that $\partial_ h(l_+(\mu), \mu) < 0$. The sign of $l_+'(\mu)$ on $J_\mu$ is the same as that of $\partial_\mu h(l_+(\mu),\mu)$, because
\begin{equation} \label{eq:hderivative}
\partial_l h(l_+(\mu), \mu) \cdot l'_+(\mu) = - \partial_\mu h(l_+(\mu),\mu).
\end{equation}
A direct calculation gives
\begin{equation} \label{eq:hmuderivative}
\partial_\mu h(l,\mu) = l\cdot  (8 (\sigma -r) (m- 2r\sigma  + r^2)-2 l (4 \mu  \sigma ^2-2 \sigma  (1+2 \mu  r)+r (3+\mu  r))).
\end{equation}
 If there exists $\mu_0 \in J_\mu$ for which $l_+(\mu_0) = 0$, then $l_+ = 0$ identically on the connected component containing $\mu_0$. In particular, the sign of $l_+$ is constant on each connected component of $J_\mu$. It is also clear from \eqref{eq:hexplicit} that
 \begin{equation} \label{eq:lasymptotics}
 l_\pm(\mu) \rightarrow 0 \text{ as } \mu \rightarrow \pm\infty.
 \end{equation}
 Consider the following cases:
 
 \begin{inparaenum} 
 \item If $m > \sigma^2$, then $l_+$ clearly attains its maximum over $I_\mu = \RR$ at some point $\mu_{\max}$. Since $l_+(0) > 0$ by the results of \S \ref{subsect:H0finitesectorproof}, we must also have $l_+(\mu_{\max}) > 0$. Furthermore, $\mu_{\max}$ is a critical point of $l_+$.
 
 \item If $m \leq \sigma^2$, then $l_+$ is nonpositive on $(\mu_+, \infty)$. In fact, if $m = \sigma^2$, then $l_+ =0$ identically on $(\mu_+, \infty)$. Thus $0 \in (-\infty,\mu_-)$, so $l_+ >0$ on $(-\infty,\mu_-)$ and $l_+$ attains its maximum at some $\mu_{\max} \in (-\infty, \mu_-]$. Actually, since $l_+'(\mu) < 0$ near $\mu_-$ (since $l_+$ is the larger of the two roots $l_\pm$), it follows that $\mu_{\max} \in (-\infty,\mu_-)$, and $\mu_{\max}$ is a critical point of $l_+$.
\end{inparaenum}

 According to \eqref{eq:hmuderivative}, there is exactly one critical point $\mu_\star$ for $l_+$, obtained by solving the simultaneous equations $\partial_l h(l,\mu) = h(l,\mu) = 0$ and using that $l_\star = l_+(\mu_\star)$ is necessarily positive:
\[
\mu_\star = \frac{m - 2 r \sigma}{2m \sigma - mr}, \quad l_\star = l_+(\mu_\star) = \frac{2m(\sigma -r)}{r} > 0.
\]
Thus we must have $\mu_{\max} = \mu_\star$, so $l_+$ achieves its global maximum over $I_\mu$ at $\mu_\star$. If $\mu_\star \leq 0$, then $l_+(\mu)$ attains its maximum on $I_\mu \cap \RR_+$ at $\mu = 0$. Note that $\mu_\star > 0$ if and only if $m > 2r \sigma$. We have thus shown that
 \[
 \sup \mathcal{F}_1' = \begin{cases} l_\star = l_+(\mu_\star) &\text{if } m > 2r\sigma, \\
 l_+(0) &\text{if } m \leq 2r\sigma.
\end{cases}
\]
Finally, $\gamma(l_\star, \mu_\star) = (m-r^2)^2 > 0$ and $\gamma(l_+(0), 0) > 0$ by the results of \S \ref{subsect:H0finitesectorproof}, so in fact $\sup \mathcal{F} = \sup \mathcal{F}'_1$.

Since this is again a finite-sector problem, the same argument as at the end of \S \ref{subsect:H0finitesectorproof} shows that Theorem \ref{theo:H0finitesectortimeinvariant} follows from Lemma \ref{lemm:popov}. \qed

\subsection{Proof of Theorem \ref{theo:H0quasi}}
Here we study exponential stability in the infinite sector $[m,\infty]$ when $f$ is $m$-quasi-strongly convex. In order to apply Lemma \ref{lemm:popov}, we must take $\mu = 0$. In that case, the frequency inequality is
\begin{equation} \label{eq:H0quasifinitesectorFDI}
\Re \left[(\lambda + \nu(i\omega + r))H(i\omega -r) \right] \leq 0.
\end{equation}
This condition is equivalent to
\[
\beta \omega^2 + \gamma  \geq 0 \text{ for all }\omega \in \RR,
\]
where the coefficients are given by
\begin{align*}
\gamma &= (\lambda + \nu r)(m - 2r\sigma + r^2), \\
\beta &= -\lambda + \nu(2\sigma - 3r).
\end{align*}
The frequency criterion reduces to $\gamma \geq 0$ and $\beta \geq 0$. We always have $\gamma \geq 0$. To ensure $\beta \geq 0$, we must have $\nu > 0$ (otherwise $\lambda = \nu = 0$, which is not a relevant case). Furthermore, the largest value of $r > 0$ for which $\beta \geq 0$ occurs when $\lambda = 0$, namely
\[
r \leq 2\sigma/3.
\]
To finish the proof, we must verify the hypotheses of Lemma \ref{lemm:popov}. It is easy to see that $\Phi(\eta,\zeta)$ is not identically when at least one of the inequalities $r \leq 2\sigma/3$ and $m \geq 2r\sigma-r^2$ is strict, since then
\[
s \mapsto \Re \left[(s + 2r)H(s) \right] 
\]
is not identically zero. In that case, Proposition \ref{prop:popovhyperstability1} even implies the existence of Lyapunov function in the form $V(z) = \langle Pz,z \rangle + g(q) - g(0)$. In the critical case $r = 2\sigma/3$ and $m = 2r\sigma-r^2$, simply let $\eta \rightarrow \infty$ and observe that
\[
\Phi(\infty,s) = (s + 2r)H(s) 
\]
is not identically zero. \qed

\subsection{Proof of Theorem \ref{theo:H0quasifinitesector}} \label{subsect:H0quasifinitesectorproof}
Theorem \ref{theo:H0finitesectortimeinvariant} leaves open the possibility that a larger sector of stability can be established when $f$ is $m$-quasi-strongly convex. In view of Theorem \ref{theo:H0quasi}, we can assume that
\[
2\sigma/3 < r < \sigma.
\]
We also take $ m> 2r\sigma -r^2$. Setting $l = L-m$ with $L < \infty$, the general frequency inequality from Lemma \ref{prop:popovhyperstability2} is
\[
\Re\left[ (\lambda + \mu(i\omega-r) +\nu(i\omega+r)H(i\omega-r) \right]  + \mu r l |H(i\omega-r)|^2  \leq \lambda/l
\]
First we find the largest allowed sector when $\lambda = 0$, and then address the problem for $\lambda > 0$.

When $\lambda = 0$, we can assume $\nu > 0$ (see \S \ref{subsect:H0finitesectorproof}), and hence rescale so that $\nu = 1$.
In that case, the frequency condition is equivalent to
\[
\beta \omega^2 + \gamma \geq 0 \text{ for all } \omega \in \RR,
\]
where this time
\begin{align*}
\gamma &= r((m - 2r\sigma + r^2) -  \mu(l + m - 2r\sigma +r^2))\\
\beta &= 2\sigma -3r + \mu(2\sigma -r).
\end{align*}
The feasible set $\mathcal{F}$ consists of $l >0$ for which $\gamma \geq 0$ and $\beta \geq 0$.
The largest allowed value of $l$ occurs when $\mu \geq  0$ is chosen to be the smallest number for which $\beta \geq 0$, namely
\begin{equation} \label{eq:mustar}
\mu_\star = \frac{3r -2\sigma}{2\sigma -r} > 0.
\end{equation}
In that case $\beta$ vanishes, and we can solve for
\begin{equation} \label{eq:bestLquasi}
\sup \mathcal{F} = \frac{4(\sigma-r)(m-2 r\sigma + r^2)}{3r - 2\sigma},
\end{equation}
We can again apply Lemma \ref{prop:popovhyperstability2}, since $\Phi(\infty,s) = ((1+\mu_\star)s + 2r)H(s)$ is not identically zero.

Next, suppose that $\lambda >0$, in which case we rescale $\lambda = 1$. The frequency condition takes the form
\begin{equation} \label{eq:H0finitesectorquasiquadratic}
\omega^4 + \beta \omega^2 + \gamma \geq 0 \text{ for all } \omega \in \RR,
\end{equation}
where the coefficients are given by
\begin{align*}
\gamma &= -\mu r l^2 + (1+r(\nu -\mu))(m-2r \sigma + r^2)l + (m-2r \sigma + r^2)^2,  \\
\beta &=  2 (\sigma^2-m) + 2 (\sigma -r)^2 -l(1 + (3r-2\sigma)\nu - (2\sigma-r)\mu). 
\end{align*} 
Let $h = 4 \gamma -\beta^2$, viewed as a function of $l, \mu, \nu$. In complete analogy with \S \ref{subsect:H0finitesectortimeinvariantproof}, define  $\mathcal{F}, \, \mathcal{F}_1,\, \mathcal{F}'_1, \, \mathcal{F}_2$, as well as their pointwise analogues $\mathcal{F}_1(\mu,\nu), \, \mathcal{F}_1'(\mu,\nu), \, \mathcal{F}_2(\mu,\nu)$.
\begin{lemm} \label{lemm:relaxed2}
$\sup \mathcal{F}'_1 \geq \sup \mathcal{F}_2$.
\end{lemm}
\begin{proof} 
It suffices to show that $\sup \mathcal{F}_2(\mu,\nu) \leq \sup \mathcal{F}'_1$ for each $\mu, \,\nu \geq 0$. If $\mu =0$, then $\gamma(l,\mu,\nu) > 0$ for all $l,\, \nu > 0$. Otherwise, if $\mu > 0$ is fixed, then $\gamma(l,\mu,\nu) \geq 0$ if and only if
\[
l \leq \gamma_1(\mu,\nu),
\]
where $\nu \mapsto \gamma_1(\mu, \nu)$ is a strictly increasing function $\nu$. Indeed, $\gamma$ is a concave quadratic function of $l$, and
\[
\disc(\gamma; l) = (1 + (\mu - \nu)^2 r^2 + 2r(\mu + \nu))(m-2r\sigma +r^2) ^2> 0. 
\]
Since $\gamma(0,\mu,\nu) = (m-2r\sigma +r^2)^2 > 0$, it follows that $l \mapsto \gamma(l, \mu, \nu)$ has two real roots, one positive and one negative. If we let $\gamma_1(\mu,\nu)$ denote the positive root, then $\nu \mapsto \gamma_1(\mu,\nu)$ is an increasing function since
\[
\partial_\nu \gamma(l,\mu, \nu) = lr(m-2 r \sigma -r^2) > 0.
\]
See the discussion surrounding \eqref{eq:hderivative} for a similar argument. Although the geometry of the zero sets of $\beta$ and $\gamma$ is different here, one can perform a case-by-case analysis similar to the proof of Lemma \ref{lemm:relaxed1}. Because this is straightforward (though more tedious in the present context), the details are omitted.
\end{proof} 

For every $\mu, \,\nu \geq 0$, the function $l \mapsto h(l,\mu,\nu)$ is again a concave quadratic. To see this, compute
\[
\partial^2_l h(l,\mu,\nu) = -2\mu ^2 (r-2 \sigma )^2 + 4 \mu  (3 r-2 \sigma ) ( (2\sigma-r)\nu - 1 )- 2(1 + (3 r -2\sigma)\nu)^2,
\]
which is a concave quadratic function of $\mu$ that is negative when $\mu = 0$ since $r > 2\sigma/3$. Furthermore,
\[
\disc(\partial^2_l h; \mu) = -64 r(2(\sigma-r) + (2\sigma-r)(3r-2\sigma)\nu), 
\]
which is therefore negative for all $\nu \geq 0$. Thus the coefficient of $l^2$ is negative for all $\mu, \, \nu \geq 0$.

By first fixing $\mu \geq 0$ and varying $\nu \geq 0$, we perform an analysis similar to the one in \S \ref{subsect:H0finitesectortimeinvariantproof}.  Since $r > 2\sigma/3$,
\[
\partial_\nu^2 \disc(h; l) = 128(m-2r\sigma + r^2)(m-4r\sigma +4r^2)(\sigma-r)^2 > 0,
\]
so this discriminant is convex with respect to $\nu$. Also,
\begin{multline*}
\disc(\disc(h; l); \nu) = 16384 (\sigma ^2-m) (\sigma -r)^4 (m-2r\sigma+r^2) \\ \times (4r^2(m-2r\sigma+r^2)\mu^2 + 4r^2(3r-2\sigma)\mu + m-2r\sigma+r^2).
\end{multline*}
The last factor is positive, so the sign of the discriminant is the same as the sign of $\sigma^2-m$ (cf. \S \ref{subsect:H0finitesectortimeinvariantproof}).

When $m\leq \sigma^2$, let $\nu_-(\mu) \leq \nu_+(\mu)$ denote the roots of $\nu \mapsto \disc(h;l)$ for a fixed $\mu \geq 0$. For a given $\mu \geq 0$, we can define sets $I_\nu(\mu) \supset J_\nu(\mu)$ in analogy with \eqref{eq:I}, \eqref{eq:J}, satisfying the following properties: 
\begin{enumerate} \itemsep6pt
	\item $l \mapsto h(l,\mu,\nu)$ has real roots $l_\pm(\mu,\nu)$ for $\nu \in I_\nu(\mu)$,
	\item $l_\pm(\mu,\nu)$ are smooth functions of $\nu$ for $\nu \in J_\nu(\mu)$.
\end{enumerate} 
Actually, if $\nu_0 \in J_\nu(\mu_0)$, then $l_\pm$ is a smooth function of $(\mu,\nu)$ in a neighborhood of $(\mu_0,\nu_0)$.
As in \eqref{eq:hderivative}, the sign of $\partial_\nu l_+(\mu,\nu)$ on $J_\nu(\mu)$ agrees with the sign of $\partial_\nu f(l_+(\mu,\nu), \mu, \nu)$. Next, compute
\begin{multline*}
\partial_\nu h(l,\mu,\nu) =
\\l \cdot (8 (\sigma -r) (m-2 r^2+(3 r -2 \sigma)\sigma)-2 l (3 r-2 \sigma ) (1 - (2\sigma-r)\mu  + (3r-2\sigma)\nu).
\end{multline*}
Thus the sign of $l_+(\mu,\nu)$ is constant on each connected component of $J_\nu(\mu)$. Since $l_+(\mu,0) > 0$ by the results of \S \ref{subsect:H0finitesectortimeinvariantproof}, it follows that $l_+(\mu,\nu) >0$ for all $\nu$ in the connected component of $I_\nu(\mu)$ containing $\nu = 0$.

 In fact, arguing exactly as in \S \ref{subsect:H0finitesectortimeinvariantproof}, one can show that $0 \in (\nu_+(\mu), \infty)$, and $\nu \mapsto l_+(\mu,\nu)$ attains its maximum over $I_\nu(\mu)$ at a point $\nu_{\max}(\mu)$ which is a critical point of $\nu \mapsto l_+(\mu,\nu)$. Any critical point $\nu_\star (\mu)$ of  $\nu \mapsto l_+(\mu,\nu)$ must satisfy
\begin{equation} \label{eq:nustar}
\nu_\star(\mu) = \frac{4 (\sigma -r) \left(m-2 r^2 + (3 r -2 \sigma)\sigma\right)}{l_\star(\mu) (3 r-2 \sigma )^2} +  \frac{(2 \sigma -r)\mu-1}{3 r-2 \sigma},
\end{equation}
where $l_\star(\mu) = l_+(\mu,\nu_\star(\mu))$. Plugging this into $h$, it follows that $l_\star(\mu)$ satisfies the equation
\begin{multline*}
4 \mu r(3r-2\sigma)^2 l^2  + 8 (3r-2\sigma) (1-2 \mu  r) (\sigma -r) (m-2r\sigma +r^2)l \\ -16 (\sigma -r)^2 (m - 4r\sigma+4r^2) (m-2r\sigma+r^2 ) = 0.
\end{multline*}
If $\mu = 0$, then there is precisely one solution to this equation, which must therefore be $l_\star(\mu)$. If $\mu > 0$, this quadratic two has real roots, the smaller of which is negative. Thus $l_\star(\mu)$ is given by the larger of the two roots. 

As a corollary, $\nu \mapsto l_+(\mu,\nu)$ has precisely one critical point, which must be $\nu_{\max}(\mu)$. Furthermore, 
\[
\partial_\nu^2 h(l_\star(\mu),\mu,\nu_\star(\mu)) = -2(3r-2\sigma)^2 l_\star^2(\mu) \neq0,
\]
so by the implicit function theorem applied to $(\mu,\nu) \mapsto \partial_\nu h(l_+(\mu,\nu),\mu,\nu)$ and uniqueness of the critical point, $\nu_\star(\mu)$ is a smooth function of $\mu$. Here we use that 
\[
\partial_\nu ( \partial_\nu h(l_+(\mu,\nu),\mu,\nu)) =  \partial_\nu^2 h(l_+(\mu,\nu),\mu,\nu)
\]
when evaluated at $\nu = \nu_\star(\mu)$, since $\partial_\nu l_+(\mu,\nu_\star(\mu))= 0$.
The function $l_\star(\mu) = l_+(\mu, \nu_\star(\mu))$ is therefore also smooth.

Next, observe that $\partial_\mu l_+(\mu,\nu_\star(\mu)) = \partial_\mu l_\star(\mu)$ since $\partial_\nu l_+(\mu, \nu_\star(\mu)) = 0$. Differentiate $(\mu,\nu) \mapsto h(l_+(\mu,\nu), \mu,\nu)$ in $\mu$ and plug \eqref{eq:nustar} into $\partial_\mu h (l_\star(\mu),\nu,\mu_\star(\mu))$ to find
\begin{equation} \label{eq:hode}
\partial_ l h(l_\star(\mu), \mu, \nu_\star(\mu)) \cdot \partial_ \mu l_\star(\mu) = 4 r \cdot l_{\star}(\mu) (l_\star(\mu) - \alpha),
\end{equation}
where we define
\[
\alpha = \frac{4(\sigma-r)(m -2r\sigma +r^2)}{3r - 2\sigma} > 0.
\]
Notice that $\alpha$ is precisely the right-hand side of \eqref{eq:bestLquasi}. Consider the following two cases.

\begin{inparaenum}
	\item First suppose that  $l_\star(0) \leq \alpha$. From \eqref{eq:hode} viewed as a differential equation satisfied by $l_\star(\mu)$, it follows that $l_\star(\mu) \leq \alpha$ for all $\mu \geq 0$. Certainly then,
	\[
	\sup \mathcal{F}_1' \leq \alpha.
	\]
	In that case, we can do no better than by choosing $\lambda, \, \mu,\,\nu \geq 0$ with $\lambda = 0$, as discussed at the beginning of this section. Now compute
	\[
	\nu_\star(0) = \frac{m - 8r^2 + 10r\sigma -4\sigma^2}{(m-4r\sigma +4r^2)(3r-2\sigma)}, \quad l_\star(0) = \frac{2(\sigma-r)(m-4r\sigma +4 r^2)}{3r - 2\sigma}.
	\]
	A direct calculation then shows
	\[
	l_\star(0) \leq \alpha\text{ if and only if } m \geq 2 r^2.
	\]
    Also note that if $l_\star(0) < \alpha$, then $l_\star(\mu) \rightarrow \alpha$ as $\mu \rightarrow \infty$. Plugging this into \eqref{eq:nustar},
	\[
	\nu_\star(\mu) \sim \frac{2\sigma -r}{3r -2\sigma}\mu
	\]
as $\mu\rightarrow \infty$, which is consistent with \eqref{eq:mustar} in the limit.

\item Next, suppose that $l_\star(0) > \alpha$. Thus, $l_\star(\mu)$ is a decreasing function of $\mu$ from the differential equation \eqref{eq:hode}. If 
\[
m \geq 8r^2 - 10r\sigma + 4\sigma^2,
\]
then $\nu_\star(0) \geq 0$, and  $\sup \mathcal{F}'_1 = l_\star(0) > \alpha$. Also,
\[
\gamma(l_\star(0), 0, \mu_\star(0)) = \frac{r^2 (m - 2r\sigma +r^2)^2}{(3 r-2 \sigma )^2} > 0,
\]
so $\sup \mathcal{F} = \sup \mathcal{F}'_1$.
Thus when $m \geq 8r^2 - 10r\sigma + 4\sigma^2$, the largest possible sector for any choice of $\mu, \,\nu, \,\lambda \geq 0$ occurs when $\mu = 0$. In that case, $r$-exponential stability follows from Lemma \ref{lemm:popov}. 
\end{inparaenum} \qed

\subsection{Proof of Theorem \ref{theo:tau>0}} \label{subsect:tau>0proof}
The proof is an application of Lemma \ref{lemm:circle}, where recall $L = \infty$. Thus the frequency inequality is just
\[
\Re H(i\omega - r) \leq 0.
\]
The frequency condition can be rewritten as
\[
\beta \omega^2 + \gamma \geq 0 \text { for all } \omega \in \RR,
\]
where the coefficients are given by
\begin{align*}
\gamma &= (r\tau -1 - 2\sigma \tau)(m - 2r\sigma + r^2 + m\tau(2\sigma  - r)), \\ 
\beta &= 1 + \tau(r-m\tau).
\end{align*} 
If $r = 2\sigma +\tau^{-1} = \sigma + \tau m /2$, then $\beta = \gamma = 0$. Furthermore, $A+\delta BC + r_\star I$ is Hurwitz for all $\delta > 0$ according to Lemma \ref{lemm:tau>0constraint}, and $\Phi(\infty,s) = H(s)$ is not identically zero. Thus Lemma \ref{lemm:circle} establishes $r_\star$-exponential stability. Moreover, although Proposition \ref{prop:popovhyperstability1} does not apply, we can instead use Lemma \ref{lemm:barabanov} to deduce the existence of quadratic Lyapunov function, since $(A,B)$ is controllable and $(A,C)$ is observable. \qed

	\bibliographystyle{alphanum}
	
	\bibliography{frequency_criteria}

\begin{thebibliography}{AABR}

\bibitem[AABR]{alvarez2002second}
Felipe Alvarez, Hedy Attouch, J{\'e}r{\^o}me Bolte, and P~Redont.
\newblock A second-order gradient-like dissipative dynamical system with
  {H}essian-driven damping.: Application to optimization and mechanics.
\newblock {\em Journal de Math{\'e}matiques Pures et Appliqu{\'e}es},
  81(8):747--779, 2002.

\bibitem[ABR]{attouch2002optimizing}
H~Attouch, J~Bolte, and P~Redont.
\newblock Optimizing properties of an inertial dynamical system with geometric
  damping: link with proximal methods.
\newblock {\em Control and Cybernetics}, 31:643--657, 2002.

\bibitem[AC]{accikmecse2008stability}
Beh{\c{c}}et A{\c{c}}{\i}kme{\c{s}}e and Martin Corless.
\newblock Stability analysis with quadratic {L}yapunov functions: Some
  necessary and sufficient multiplier conditions.
\newblock {\em Systems \& control letters}, 57(1):78--94, 2008.

\bibitem[ACFR]{attouch2019first}
Hedy Attouch, Zaki Chbani, Jalal Fadili, and Hassan Riahi.
\newblock First-order optimization algorithms via inertial systems with
  {H}essian driven damping.
\newblock {\em arXiv preprint arXiv:1907.10536}, 2019.

\bibitem[ACPR]{attouch2018fast}
Hedy Attouch, Zaki Chbani, Juan Peypouquet, and Patrick Redont.
\newblock Fast convergence of inertial dynamics and algorithms with asymptotic
  vanishing viscosity.
\newblock {\em Mathematical Programming}, 168(1-2):123--175, 2018.

\bibitem[APR1]{attouch2014dynamical}
H{\'e}dy Attouch, Juan Peypouquet, and Patrick Redont.
\newblock A dynamical approach to an inertial forward-backward algorithm for
  convex minimization.
\newblock {\em SIAM Journal on Optimization}, 24(1):232--256, 2014.

\bibitem[APR2]{attouch2016fast}
Hedy Attouch, Juan Peypouquet, and Patrick Redont.
\newblock Fast convex optimization via inertial dynamics with {H}essian driven
  damping.
\newblock {\em Journal of Differential Equations}, 261(10):5734--5783, 2016.

\bibitem[Bli]{bliman2002absolute}
Pierre-Alexandre Bliman.
\newblock Absolute stability criteria with prescribed decay rate for
  finite-dimensional and delay systems.
\newblock {\em Automatica}, 38(11):2015--2019, 2002.

\bibitem[BLR]{boczar2015exponential}
Ross Boczar, Laurent Lessard, and Benjamin Recht.
\newblock Exponential convergence bounds using integral quadratic constraints.
\newblock In {\em 2015 54th IEEE Conference on Decision and Control (CDC)},
  pages 7516--7521. IEEE, 2015.

\bibitem[CS]{carrasco2018conditions}
Joaquin Carrasco and Peter Seiler.
\newblock Conditions for the equivalence between {IQC} and graph separation
  stability results.
\newblock {\em International Journal of Control}, pages 1--8, 2018.

\bibitem[HM1]{hill1976stability}
David Hill and Peter Moylan.
\newblock The stability of nonlinear dissipative systems.
\newblock {\em IEEE Transactions on Automatic Control}, 21(5):708--711, 1976.

\bibitem[HM2]{hill1980dissipative}
David~J Hill and Peter~J Moylan.
\newblock Dissipative dynamical systems: Basic input-output and state
  properties.
\newblock {\em Journal of the Franklin Institute}, 309(5):327--357, 1980.

\bibitem[HS]{hu2016exponential}
Bin Hu and Peter Seiler.
\newblock Exponential decay rate conditions for uncertain linear systems using
  integral quadratic constraints.
\newblock {\em IEEE Transactions on Automatic Control}, 61(11):3631--3637,
  2016.

\bibitem[Isi]{isidori2013nonlinear}
Alberto Isidori.
\newblock {\em Nonlinear control systems}.
\newblock Springer Science \& Business Media, 2013.

\bibitem[Meg]{megretski2010kyp}
Alexandre Megretski.
\newblock {KYP} lemma for non-strict inequalities and the associated minimax
  theorem.
\newblock {\em arXiv preprint arXiv:1008.2552}, 2010.

\bibitem[Mol]{molinari1975conditions}
B~Molinari.
\newblock Conditions for nonpositive solutions of the linear matrix inequality.
\newblock {\em IEEE Transactions on Automatic Control}, 20(6):804--806, 1975.

\bibitem[Moy]{moylan1975frequency}
P~Moylan.
\newblock On a frequency-domain condition in linear optimal control theory.
\newblock {\em IEEE Transactions on Automatic Control}, 20(6):806--806, 1975.

\bibitem[MR]{megretski1997system}
Alexandre Megretski and Anders Rantzer.
\newblock System analysis via integral quadratic constraints.
\newblock {\em IEEE Transactions on Automatic Control}, 42(6):819--830, 1997.

\bibitem[NNG]{necoara2019linear}
Ion Necoara, Yu~Nesterov, and Francois Glineur.
\newblock Linear convergence of first order methods for non-strongly convex
  optimization.
\newblock {\em Mathematical Programming}, 175(1-2):69--107, 2019.

\bibitem[PG]{popov1973hyperstability}
VM~Popov and Radu Georgescu.
\newblock {\em Hyperstability of control systems}.
\newblock Springer-Verlag, 1973.

\bibitem[RV]{reis2015kalman}
Timo Reis and Matthias Voigt.
\newblock The {K}alman--{Y}akubovich--{P}opov inequality for
  differential-algebraic systems: Existence of nonpositive solutions.
\newblock {\em Systems \& Control Letters}, 86:1--8, 2015.

\bibitem[SA]{safonov1979stability}
Michael~G Safonov and Michael Athans.
\newblock On stability theory.
\newblock In {\em 1978 IEEE Conference on Decision and Control including the
  17th Symposium on Adaptive Processes}, pages 301--314. IEEE, 1979.

\bibitem[SDJS]{shi2018understanding}
Bin Shi, Simon~S Du, Michael~I Jordan, and Weijie~J Su.
\newblock Understanding the acceleration phenomenon via high-resolution
  differential equations.
\newblock {\em arXiv preprint arXiv:1810.08907}, 2018.

\bibitem[SV]{scherer2018stability}
Carsten~W Scherer and Joost Veenman.
\newblock Stability analysis by dynamic dissipation inequalities: On merging
  frequency-domain techniques with time-domain conditions.
\newblock {\em Systems \& Control Letters}, 121:7--15, 2018.

\bibitem[TR]{trentelman2001pick}
Harry~L Trentelman and Paolo Rapisarda.
\newblock Pick matrix conditions for sign-definite solutions of the algebraic
  {R}iccati equation.
\newblock {\em SIAM journal on control and optimization}, 40(3):969--991, 2001.

\bibitem[Wil1]{willems1971least}
Jan Willems.
\newblock Least squares stationary optimal control and the algebraic {R}iccati
  equation.
\newblock {\em IEEE Transactions on Automatic Control}, 16(6):621--634, 1971.

\bibitem[Wil2]{willems1974existence}
Jan Willems.
\newblock On the existence of a nonpositive solution to the {R}iccati equation.
\newblock {\em IEEE Transactions on Automatic Control}, 19(5):592--593, 1974.

\bibitem[Wil3]{willems1972dissipativeI}
Jan~C Willems.
\newblock Dissipative dynamical systems part {I}: General theory.
\newblock {\em Archive for rational mechanics and analysis}, 45(5):321--351,
  1972.

\bibitem[Wil4]{willems1972dissipativeII}
Jan~C Willems.
\newblock Dissipative dynamical systems part {II}: Linear systems with
  quadratic supply rates.
\newblock {\em Archive for rational mechanics and analysis}, 45(5):352--393,
  1972.

\bibitem[WT]{willems1998quadratic}
JC~Willems and HL~Trentelman.
\newblock On quadratic differential forms.
\newblock {\em SIAM Journal on Control and Optimization}, 36(5):1703--1749,
  1998.

\end{thebibliography}

\end{document}